\documentclass[12pt]{amsart}
\usepackage{amssymb}
\usepackage{amsfonts}
\usepackage{latexsym}
\usepackage{amscd}
\usepackage[mathscr]{euscript}
\usepackage{xy} \xyoption{all}
\usepackage[active]{srcltx}

\usepackage{tikz}

\usepackage{color,graphics}

\newcommand{\Z}{{\mathbb{Z}}}

\newcommand{\C}{{\mathbb{C}}}

\newcommand{\uloopr}[1]{\ar@'{@+{[0,0]+(-4,5)}@+{[0,0]+(0,10)}@+{[0,0] +(4,5)}}^{#1}}
\newcommand{\uloopd}[1]{\ar@'{@+{[0,0]+(5,4)}@+{[0,0]+(10,0)}@+{[0,0]+ (5,-4)}}^{#1}}
\newcommand{\dloopr}[1]{\ar@'{@+{[0,0]+(-4,-5)}@+{[0,0]+(0,-10)}@+{[0, 0]+(4,-5)}}_{#1}}
\newcommand{\dloopd}[1]{\ar@'{@+{[0,0]+(-5,4)}@+{[0,0]+(-10,0)}@+{[0,0 ]+(-5,-4)}}_{#1}}

\newcommand{\calV}{{\mathcal V}}

\newcommand{\gotr}{{\mathfrak r}}

\newcommand{\luloop}[1]{\ar@'{@+{[0,0]+(-8,2)}@+{[0,0]+(-10,10)}@+{[0, 0]+(2,2)}}^{#1}}

\newcommand{\dotedge}{\ar@{.}}
\newcommand{\eqedge}{\ar@{=}}

\newcommand{\FSGr}{\mathbf{FSGr}}
\newcommand{\BFSGr}{\mathbf{BFSGr}}

\DeclareMathOperator{\coker}{coker}

\newcommand{\mon}[1]{\calV(#1)} 

\newcommand{\So}{{\rm Source}}

\newcommand{\Cstaralg}{C^*\text{-}\mathbf{alg}}

\newcommand{\bgast}{\mbox{\Large$*$}}

\numberwithin{equation}{section}

\theoremstyle{plain}
\newtheorem{theorem}{Theorem}[section]
\newtheorem{lemma}[theorem]{Lemma}
\newtheorem{proposition}[theorem]{Proposition}
\newtheorem{corollary}[theorem]{Corollary}

\theoremstyle{definition}
\newtheorem{definition}[theorem]{Definition}
\newtheorem{example}[theorem]{Example}

\newtheorem{remark}[theorem]{Remark}
\newtheorem*{remark*}{Remark}

\newtheorem*{assumption*}{Assumption}

\newtheorem{construction}[theorem]{Construction}

\newtheorem{notation}[theorem]{Notation}

\begin{document}

\parindent 20pt

\null\vskip-1cm

\title[K-theory]{$K$-theory for the tame C*-algebra of a separated graph}%
\author{Pere Ara}
\address{Departament de Matem\`atiques, Universitat Aut\`onoma de Barcelona,
08193 Bellaterra (Barcelona), Spain.} \email{para@mat.uab.cat}
\author{ Ruy Exel}
\address{Departamento de Matem\'atica, Universidade Federal de Santa Catarina,
88010-970 Florian\'opolis SC, Brazil.}\email{exel@mtm.ufsc.br}
\thanks{The first named author was partially supported by DGI MICIIN-FEDER
MTM2011-28992-C02-01, and by the Comissionat per Universitats i
Recerca de la Generalitat de Catalunya. The second named author was
partially supported by CNPq.} 
\date{\today}

\maketitle

\begin{abstract} A {\it separated graph} is a pair $(E,C)$ consisting of a directed
graph $E$ and a set $C=\bigsqcup _{v\in E^0} C_v$, where each $C_v$
is a partition of the set of edges whose terminal vertex is $v$.
Given a separated graph $(E,C)$, such that all the sets $X\in C$ are finite, the K-theory of the graph C*-algebra
$C^*(E,C)$ is known to be determined by the kernel and the cokernel of a certain map, denoted by $1_C- A_{(E,C)}$, from $\Z^{(C)}$ to $\Z^{(E^0)}$.
In this paper, we compute the K-theory of the {\it tame} graph C*-algebra $\mathcal O (E,C)$ associated to $(E,C)$,
which has been recently introduced by the authors. Letting $\pi$ denote the natural surjective homomorphism from
$C^*(E,C)$ onto $\mathcal O (E,C)$, we show that $K_1(\pi)$ is a group isomorphism, and that $K_0(\pi)$ is a split monomorphism, whose cokernel is
a torsion-free abelian group. We also prove that this cokernel is a free abelian group when the graph $E$ is finite, and determine its generators in terms
of a sequence of separated graphs $\{ (E_n, C^n) \}_{n=1}^{\infty}$ naturally attached to $(E,C)$.
On the way to showing our main results, we obtain an explicit  description of a connecting map  arising in a six-term exact sequence computing the
K-theory of an amalgamated free product, and we also exhibit an explicit isomorphism between $\ker (1_C - A_{(E,C)})$ and
$K_1(C^*(E,C))$.
\end{abstract}

\section{Introduction}

A {\it separated graph} is a pair $(E,C)$ consisting of a directed
graph $E$ and a set $C=\bigsqcup _{v\in E^0} C_v$, where each $C_v$
is a partition of the set of edges whose terminal vertex is $v$.
Their associated C*-algebras $C^*(E,C)$ (\cite{AG2}, \cite{Aone-rel}) provide
generalizations of the usual graph C*-algebras (see e.g. \cite{Raeburn}) associated to
directed graphs, although these algebras behave quite differently
from the usual graph algebras because the range projections corresponding to different edges need not
commute. One motivation for
their introduction was to provide graph-algebraic models for the C*-algebras
$U^{\text{nc}}_{m,n}$ studied  by L. Brown \cite{Brown} and  McClanahan
 \cite{McCla1}, \cite{McCla2}, \cite{McCla3}. Another motivation
was to obtain graph C*-algebras whose structure of projections is as
general as possible.  The theory of \cite{AG2} was mainly developed for {\it finitely separated graphs}, which are those
separated graphs
$(E,C)$ such that all the sets $X\in C$ are finite.

\medskip

Recall that a set $S$ of partial isometries in a C*-algebra
$\mathcal A$ is said to be {\it tame} \cite[Proposition
5.4]{ExelPJM} if every element of $U=\langle S\cup S^*\rangle$, the
multiplicative semigroup generated by $S\cup S^*$, is a partial
isometry. As indicated above,
a main difficulty in working with $C^*(E,C)$ is that, in general, the generating set of partial
isometries of these algebras is not tame. This is not the case for
the usual graph algebras, where it can be easily shown that the
generating set of partial isometries is tame. In order to solve this problem,
we introduced in \cite{AE} the {\it tame graph
C*-algebra} $\mathcal O (E,C)$ of a separated graph. Roughly, this algebra is defined by imposing
to $C^*(E,C)$ the relations needed to transform  the canonical  generating set of partial
isometries into a tame set of partial isometries (see Section \ref{sect:prels} for the precise definitions).

\medskip

For a finite bipartite separated graph $(E,C)$, a dynamical interpretation of the C*-algebra $\mathcal O (E,C)$
was obtained in \cite{AE}, and using this, a useful representation of $\mathcal O (E,C)$ as a {\it partial crossed product}
of a commutative C*-algebra by a finitely generated free group was derived. This theory enabled the authors to solve (\cite[Section 7]{AE}) an open problem
on paradoxical decompositions in a topological setting, posed in \cite{KN} and \cite{RS}.
It is worth mentioning here that the restriction to bipartite graphs in this
theory is harmless, since by \cite[Proposition 9.1]{AE}, we can attach to every separated graph $(E,C)$ a bipartite separated graph $(\tilde{E},\tilde{C})$
in such a way that the respective (tame) graph C*-algebras are Morita-equivalent.

One of the main technical tools in \cite{AE} is the introduction, for each finite bipartite separated graph $(E,C)$,
of a sequence of finite bipartite separated graphs $\{ (E_n, C^n)\}$ such that
the graph C*-algebras $C^*(E_n,C^n)$ approximate the tame graph C*-algebra $\mathcal O (E, C)$, in the sense that
$\mathcal O (E, C) \cong \varinjlim _n C^*(E_n,C^n)$, see \cite[Section 5]{AE}.

The main purpose of this paper is to compute the K-theory of the tame graph C*-algebras of finitely separated graphs. Concretely, we show the following
result:

\begin{theorem}
\label{thm:BIGONE}
 Let $(E,C)$ be a finitely separated graph. Then
 \begin{enumerate}
  \item $K_0(\mathcal O (E,C)) \cong K_0(C^*(E,C))\bigoplus H \cong \coker (1_C-A_{(E,C)})\bigoplus H,$
  where $H$ is a torsion-free abelian group. The group $H$ is a free abelian group when $E$ is a finite graph.
  \item The canonical projection map $\pi \colon C^*(E,C)\to \mathcal O (E,C)$ induces an isomorphism
  $$K_1(\mathcal O (E,C))\cong K_1(C^*(E,C))\cong \ker (1_C-A_{(E,C)}).$$
 \end{enumerate}
 \end{theorem}

The terms $\coker (1_C-A_{(E,C)})$ and $\ker (1_C-A_{(E,C)})$ appearing in the above theorem come from
  \cite[Theorem 5.2]{AG2},
  where the K-theory of the graph C*-algebras of finitely separated graphs was computed. The formulas there are
analogous to the ones previously known for non-separated graphs (see \cite[Theorem 3.2]{RaeSzy}). The matrix $A_{(E,C)}$ is
the incidence matrix of the separated graph, which encodes the number of edges between two vertices of $E$
  belonging
  to the different sets $X\in C$.  (See Section \ref{sect:compK1} for the precise definition of these
matrices).

We first study the case of finite bipartite separated graphs. Under this additional hypothesis, we obtain the result
for $K_0$ in Section \ref{sect:bipsepgraphs} (Theorem \ref{thm:K0C(E,C)})
and the result for $K_1$ in Section  \ref{sect:compK1} (Theorem \ref{thm:main-K1-first}).
The proof of Theorem \ref{thm:main-K1-first} involves a computation of the index map for certain amalgamated free products, which we develop in
Section \ref{sect:partialunitaries}. As a
    byproduct
  of our approach, we also develop a concrete description of the isomorphism between $\ker (1_C-A_{(E,C)})$ and
  $K_1(C^*(E,C))$,
which we believe is of independent interest. Such a description was obtained by Carlsen, Eilers and
Tomforde in \cite[Section 3]{CET} for relative graph algebras of non-separated graphs, by using different
techniques. Using these results and direct limit technology, we show Theorem \ref{thm:BIGONE} in Section
\ref{sect:reduction-lemmas}
 (see Theorems \ref{thm:mainK1fggrs} and \ref{thm:finalK0}).

\medskip

\noindent{\bf Contents.} We now explain in more detail the contents
of this paper. In Section 2 we recall the basic definitions needed
for our work, coming from the papers \cite{AG2}, \cite{AG} and \cite{AE}.
In Section 3, we recall the crucial concept
of a multiresolution of a separated graph $(E,C)$ at a set of
vertices of $E$, and we determine the precise relation between the correponding
graph C*-algebras (Lemma \ref{lem:K0multires}). This is a
vital step for our results on $K_0$. In Section \ref{sect:bipsepgraphs}, we show the
isomorphism $K_0(\mathcal O (E,C))\cong K_0(C^*(E,C))\oplus H$ for any finite bipartite separated graph $(E,C)$,
where $H$ is a free abelian group, generally of infinite rank. The generators of $H$ are
precisely determined in terms of the vertices of the graphs appearing in the canonical sequence
$\{ (E_n, C^n) \}$ of finite bipartite separated graphs
associated to $(E,C)$ (see Theorem \ref{thm:K0C(E,C)}).
Section \ref{sect:partialunitaries} contains the explicit calculation of the index map
$K_1(A_1*_BA_2)\to K_0(B)$ of \cite{Thomsen} for certain partial unitaries in the amalgamated
free product $A_1*_BA_2$, where $B$ is a finite-dimensional C*-algebra and $A_1$, $A_2$ are separable C*-algebras.
This result is used in Section \ref{sect:compK1}, where the isomorphism between
$K_1(\mathcal O (E,C))$ and $K_1(C^*(E,C))$ is obtained for any finite bipartite separated graph $(E,C)$.
We obtain indeed an enhanced version of this result (Theorem \ref{thm:main-K1-first}), which includes an
explicit isomorphism of the above mentioned groups with the group $\ker (1_C-A_{(E,C)})$.
We also show a corresponding result for the {\it reduced} tame graph C*-algebra $\mathcal O _{red} (E,C)$
(Corollary \ref{cor:reduced-version}).
Finally, we extend the above results to
    (not necessarily bipartite)
  finitely separated graphs in Section \ref{sect:reduction-lemmas}.
For this, we use the direct limit technology of \cite{AG} and \cite[Proposition 9.1]{AE}.
 The result for $K_1$ is easily derived using these techniques (Theorem \ref{thm:mainK1fggrs}).
 To obtain the result for $K_0$, we need to refine some of the already developed tools, in particular we
 make use of the concrete information about the generators of the cokernel of the map
 $K_0(\pi) \colon K_0(C^*(E,C))\to K_0(\mathcal O (E,C))$ induced by the canonical surjection
 $\pi \colon C^*(E,C)\to \mathcal O (E,C)$ for finite bipartite separated graphs,
 see Theorems \ref{thm:mainK0fggrs} and \ref{thm:finalK0}.

\section{Preliminary definitions}
\label{sect:prels}

The concept of separated graph, introduced in \cite{AG}, plays a
vital role in our construction. In this section, we will recall this
concept and we will also recall the definitions of the {\it monoid
associated to a separated graph}, the {\it Leavitt path algebra} and
the {\it graph C*-algebra} of a separated graph.

  Regarding the direction of arrows in graphs, we will use notation opposite that of \cite{AG} and \cite{AG2}, but in
agreement with the one used in \cite{AEK}, and in the book \cite{Raeburn}.

\begin{definition}{\rm (\cite{AG})} \label{defsepgraph}
A \emph{separated graph} is a pair $(E,C)$ where $E$ is a graph,
$C=\bigsqcup _{v\in E^ 0} C_v$, and $C_v$ is a partition of
$r^{-1}(v)$ (into pairwise disjoint nonempty subsets) for every
vertex $v$. (In case $v$ is a source, we take $C_v$ to be the empty
family of subsets of $r^{-1}(v)$.)

    If all sets
  in $C$ are finite, we say that $(E,C)$ is a
\emph{finitely separated} graph. This necessarily holds if $E$ is
column-finite (that is, if $r^{-1}(v)$ is a finite set for every
$v\in E^0$.)

The set $C$ is a \emph{trivial separation} of $E$ in case $C_v=
\{r^{-1}(v)\}$ for each $v\in E^0\setminus \So (E)$. In that case,
$(E,C)$ is called a \emph{trivially separated graph} or a
\emph{non-separated graph}.
\end{definition}

\begin{definition}\cite[Definition 1.4]{AG2}
\label{def:LPASG} {\rm The
{\it Leavitt path algebra of the separated graph} $(E,C)$ is the $*$-algebra $L_\C(E,C)$ with
generators $\{ v, e\mid v\in E^0, e\in E^1 \}$, subject to the
following relations:}
\begin{enumerate}
\item[] (V)\ \ $vv^{\prime} = \delta_{v,v^{\prime}}v$ \ and \ $v=v^*$ \ for all $v,v^{\prime} \in E^0$,
\item[] (E)\ \ $r(e)e=es(e)=e$ \ for all $e\in E^1$ ,
\item[] (SCK1)\ \ $e^*e'=\delta _{e,e'}s(e)$ \ for all $e,e'\in X$, $X\in C$, and
\item[] (SCK2)\ \ $v=\sum _{ e\in X }ee^*$ \ for every finite set $X\in C_v$, $v\in E^0$.
\end{enumerate}
\end{definition}

We now recall the definition of the graph C*-algebra $C^*(E,C)$,
introduced in \cite{AG2}.

\begin{definition}\cite[Definition 1.5]{AG2} The \emph{graph C*-algebra} of a separated graph $(E,C)$ is the
C*-algebra $C^*(E,C)$  with generators $\{ v, e \mid v\in E^0,\ e\in
E^1 \}$, subject to the relations (V), (E), (SCK1), (SCK2). In other
words, $C^*(E,C)$ is the enveloping C*-algebra of $L_\C(E,C)$.
\end{definition}

In case $(E,C)$ is trivially separated, $C^*(E,C)$ is just the
classical graph C*-algebra $C^*(E)$. There is a unique
*-homomorphism  $L_\C(E,C) \rightarrow
C^*(E,C)$ sending the generators of $L_\C(E,C)$ to their canonical
images in $C^*(E,C)$. This map is injective by \cite[Theorem
3.8(1)]{AG2}.

The C*-algebra $C^*(E,C)$ for separated graphs behaves in quite a
different way compared to the usual graph C*-algebras associated to non-separated graphs, the reason being that the final projections
of the partial isometries corresponding to edges coming from different sets in
$C_v$, for $v\in E^0$, need not commute. In order to resolve this problem, a different C*-algebra was considered in
\cite{AE}, as follows:

\begin{definition} \cite{AE}
 \label{def:O(E,C)} Let $(E,C)$ be any separated graph. Let $U$ be the multiplicative subsemigroup of $C^*(E,C)$ generated by $(E^1)\cup (E^1)^*$ and write $e(u)=uu^*$
 for $u\in U$.
 Then the {\it tame graph C*-algebra} of $(E,C)$ is the C*-algebra
 $$\mathcal O (E,C)= C^*(E,C)/ J \, ,$$
 where $J$ is the closed ideal of $C^*(E,C)$ generated by all the commutators $[e(u), e(u')]$, for $u,u'\in U$.
 \end{definition}

Observe that $J=0$ in the non-separated case, so we get that $\mathcal O (E)= C^*(E)$ is the usual graph C*-algebra in this case.

Recall that for a unital ring $R$, the monoid $\mon{R}$ is usually
defined as the set of isomorphism classes $[P]$ of finitely
generated projective (left, say) $R$-modules $P$, with an addition
operation given by $[P]+[Q]= [P\oplus Q]$. For a nonunital version,
see  \cite[Definition 10.8]{AG}.

For arbitrary rings, $\mon{R}$ can also be described in terms of
equivalence classes of idempotents from the ring $M_\infty(R)$ of
  all infinite  matrices over $R$ with finitely many nonzero
entries. The equivalence relation is \emph{Murray-von Neumann
equivalence}: idempotents $e,f\in M_\infty(R)$ satisfy $e\sim f$ if
and only if there exist $x,y\in M_\infty(R)$ such that $xy=e$ and
$yx=f$. Write $[e]$ for the equivalence class of $e$; then $\mon{R}$
can be identified with the set of these classes. Addition in
$\mon{R}$ is given by the rule $[e]+[f]= [e\oplus f]$, where
$e\oplus f$ denotes the block diagonal matrix $\left(
\begin{smallmatrix} e&0\\ 0&f \end{smallmatrix} \right)$. With this
operation, $\mon{R}$ is a commutative monoid, and it is \emph{conical},
meaning that $a+b=0$ in $\mon{R}$ only when $a=b=0$. Whenever $A$ is
a C*-algebra, the monoid $\mon{A}$ agrees with the monoid of
equivalence classes of projections  in $M_{\infty}(A)$ with respect
to the equivalence relation given by $e\sim f $ if and only if there
is a partial isometry $w$ in $M_{\infty}(A)$ such that $e=ww^*$ and
$f=w^*w$; see \cite[4.6.2 and 4.6.4]{Black} or \cite[Exercise
3.11]{rordam}.

We will need the definition of $M(E,C)$ only for finitely separated
graphs. The reader can consult \cite{AG} for the definition in the
general case. Let $(E,C)$ be a finitely separated graph, and let
$M(E,C)$ be the commutative monoid given by generators $a_v$, $v\in
E^0$, and relations $a_v=\sum _{e\in X} a_{s(e)}$, for $X\in C_v$,
$v\in E^0$. Then there is a canonical monoid homomorphism $M(E,C)\to
\mon{L_\C(E,C)}$, which is shown to be an isomorphism in
\cite[Theorem 4.3]{AG}. The map $\mon{L_{\C}(E,C)}\to
\mon{C^*(E,C)}$ induced by the natural $*$-homomorphism $L_\C(E,C)\to
C^*(E,C)$ is conjectured to be an isomorphism for all finitely
separated graphs $(E,C)$ (see \cite{AG2} and \cite[Section
6]{Aone-rel}).

\section{Multiresolutions}
\label{sect:multires}

In this section, we will recall from \cite{AE} the concept of mutiresolution of a
finitely separated graph $(E,C)$, which is
closely related to the notion of resolution, studied in \cite{AG}. We will also establish the precise relation between
the corresponding Grothendieck groups.

\begin{definition}{\rm (\cite{AE})}
\label{multiresatv}
  Let $(E,C)$ be a finitely separated graph, and let $v$ be any given vertex.
Let $C_v=\{ X_1,\dots ,X_k\}$ with each $X_i$ a
finite subset of $r^{-1}(v)$. Put $M=\prod _{i=1}^k |X_i|$. Then the
multiresolution of $(E,C)$ at $v$ is the separated graph $(E_v,
C^v)$ with
$$E_v^0=E^0\sqcup \{v(x_1,\dots ,x_k)\mid x_i\in X_i, i=1,\dots ,k \},$$
and with $E_v^1=E^1\sqcup \Lambda$, where $\Lambda $ is a new set of
arrows defined as follows. For each $x_i\in X_i$, we put $M/|X_i|$
new arrows
  $\alpha ^{x_i}(x_1,\dots ,x_{i-1},x_{i+1},\dots ,x_k )$,
$x_j\in X_j$, $j\ne i$, with
$$r(\alpha ^{x_i}(x_1,\dots ,x_{i-1},x_{i+1},\dots ,x_k ))=s(x_i), \text{ and }
s(\alpha ^{x_i}(x_1,\dots ,x_{i-1},x_{i+1},\dots ,x_k ))=v(x_1,\dots
,x_k).$$ For a vertex $w\in E^0$, define the new groups at $w$ as
follows. These groups are indexed by the edges $x_i\in X_i$,
$i=1,\dots ,k$, such that $s(x_i)=w$. For
    each such $x_i$, set
$$X(x_i) =\{\alpha ^{x_i} (x_1,\dots ,x_{i-1},x_{i+1},\dots , x_k)\mid x_j\in X_j, j\ne i\}.$$
Then $$(C^v)_w=C_w\sqcup \{X(x_i)\mid x_i\in X_i, s(x_i)=w,
i=1,\dots , k\}.$$ The new vertices $v(x_1,\dots , x_k)$ are sources
in $E_v$.
\end{definition}

\begin{definition} {\rm (\cite{AE})}
\label{multiresatsetofvs} Let $V\subseteq E^0$ be a set of vertices
such that, for each $u\in V$,  $C_u=\{X_1^u, \dots ,X_{k_u}^u\}$,
with each $X^u_i$ a finite subset of $r^{-1}(u)$. Then the {\it
multiresolution of $(E,C)$ at $V$} is the separated graph $(E_V, C^V)$
obtained by applying the above process to all vertices $u$ in $V$.

Hence
$$E_V^0=E^0\sqcup \Big( \bigsqcup _{u\in V} \{v(x^u_1,\dots ,x^u_{k_u})\mid x^u_i\in X^u_i, i=1,\dots ,k_u \} \Big),$$
and $E_V^1=E^1\sqcup \Big( \bigsqcup _{u\in V} \Lambda_u\Big)$,
where $\Lambda _u$ is the corresponding set of arrows, defined as in Definition \ref{multiresatv}, for each $u\in V$.
The sets $(C^V)_w$, for $w\in E_V^0$, are defined just as in Definition
\ref{multiresatv}:
$$(C^V)_w=C_w\sqcup \{X(x^u_i)\mid x^u_i\in X^u_i, s(x^u_i)=w,
i=1,\dots , k_u, u\in V\}.$$ The new vertices $v(x^u_1,\dots ,
x^u_{k_u})$ are sources in $E_V$.
\end{definition}

We will only need to consider multiresolutions at sets of vertices $V$ such that there are no edges between them. Observe that this implies that $r_E(v)= r_{E_V}(v)$
for all $v\in V$.

The notation used in the next
  lemma
  will become clear when we prove
Lemma \ref{lem:K0multires}.

\begin{lemma}
\label{lem:refin-exactsequence} Let $X_1,\dots ,X_k$ be $k$ finite
sets, with $X_i= \{ x^{(i)}_t\}_{t=1,\dots , |X_i|}$ and let $G(M)$
be the abelian group generated by the $|X_1|+\cdots +|X_k|$ elements
$$\{ b(x^{(i)}_t)\mid t=1,\dots , |X_i|, i=1,\dots ,k \}$$
subject to the relations $\sum _{t=1}^{|X_i|}  b(x^{(i)}_t) -
\sum_{s=1}^{|X_j|} b(x^{(j)}_s)=0$, for $1\le i<j\le k$. Let $G(F)$
be the free abelian group on the $|X_1|\cdot |X_2|\cdots |X_k|$
elements $a(x^{(1)}_{t_1},x^{(2)}_{t_2},\dots , x^{(k)}_{t_k})$, for
$t_i\in \{ 1,\dots ,|X_i| \}$, $i\in \{ 1,\dots ,k \}$. Let $G(\psi
)\colon G(M)\to G(F)$ be the group homomorphism given by
\begin{equation}
\label{eq:Gofpsi}
G(\psi)(b(x^{(i)}_{t_i}))= \sum _{j\ne i} \sum_{t_j=1}^{|X_j|} a(x^{(1)}_{t_1},\dots ,x^{(i-1)}_{t_{i-1}}, x^{(i)}_{t_i},x^{(i+1)}_{t_{i+1}},\dots ,x^{(k)}_{t_k})
\end{equation}
 for $1\le t_i\le |X_i|$,
$1\le i\le k$.  Then $G(\psi)$ is injective, and
$G(F)/G(\psi)(G(M))$ is a free abelian group of rank $|X_1|\cdots
|X_k|-|X_1|-\cdots -|X_k|+k-1$, freely generated by the images in
$G(F)/G(\psi)(G(M))$ of the elements of the form
$a(x^{(1)}_{t_1},x^{(2)}_{t_2},\dots , x^{(k)}_{t_k})$ such that
$t_i>1$ and $t_j>1$ for at least two distinct indices $i,j\in \{
1,\dots ,k\}$.
\end{lemma}

\begin{proof}
Observe that $G(\psi)$ is a well-defined homomorphism, since
$G(\psi)$ sends $\sum _{t=1}^{|X_i|}  b(x^{(i)}_t) -
\sum_{s=1}^{|X_j|} b(x^{(j)}_s)$ to $0$ for all $i,j$.

It is easy to check that
$$\mathfrak{B} = \{ b(x^{(1)}_{t_1}) \mid 1\le t_1\le |X_1| \}\cup \{
b(x^{(i)}_{t_i})\mid 2\le t_i\le |X_i|, \,\, 2\le i\le k \}$$ is a family
of generators for $G(M)$, with $|X_1|+\cdots +|X_k|-k+1$ elements.

Write
$$B^{(1)}_1= G(\psi) (b(x^{(1)}_1)),\qquad B^{(i)}_{t_i}= G(\psi)
(b(x^{(i)}_{t_i})),\quad \, 2\le t_i \le |X_i|,\,  1\le i \le k \,
.$$ Let $\mathcal B$ be the canonical basis of $G(F)$, and let
$\mathcal B '$ be the subset of elements of $\mathcal B$ which are
of the form $a(x^{(1)}_{t_1},x^{(2)}_{t_2},\dots , x^{(k)}_{t_k})$
with $t_i>1$ and $t_j>1$ for at least two distinct indices $i,j\in
\{ 1,\dots ,k\}$. By using the integer version of Steinitz's Lemma,
we see that $$\{B^{(1)}_1\}\cup \Big(\mathcal B \setminus
\{a(x^{(1)}_1,x^{(2)}_1, \dots , x^{(k)}_1)\}\Big) $$ is a basis for
$G(F)$.

Now observe that for all $i\in \{ 1,\dots , k\}$ and $t_i\in \{
2,\dots ,|X_i| \}$, we have
$$ B^{(i)}_{t_i}\in a(x^{(1)}_1,\dots x^{(i-1)}_{1},
x^{(i)}_{t_i},x^{(i+1)}_1,\dots , x^{(k)}_1) + \langle \mathcal B '
\rangle .$$ Hence, the integer version of Steinitz's Lemma gives
immediately that
$$ \{ B^{(1)}_1 \}\cup \{ B^{(i)}_{t_i} \mid 2\le t_i \le |X_i|, 1\le
i\le k  \}\cup \mathcal B '$$ is a basis of $G(F)$. This shows in
particular that $G(\psi)$ is injective and that the above generating
family $\mathfrak B$ is a basis for $G(M)$. It also shows that
$\mathcal B '$ is a free basis for $G(F)/G(\psi)(G(M))$.
\end{proof}

We will use the following lemma
   to compute $K_0(\mathcal O (E,C))$.

\begin{lemma} \label{lem:K0multires}
Let $(E,C)$ be a separated graph and let $V\subseteq E^0$ be a finite set
of vertices such that $|r^{-1}(u)|<\infty$ for all $u\in V$. Suppose that $s(r^{-1}(V))\cap V=\emptyset$,
that is, that there are no edges between
elements of $V$.
For
$u\in V$, set $C_u= \{X^u_1,\dots X^u_{k_u} \}$. Let $\iota: (E,C)
\rightarrow (E_V,C^V)$ denote the inclusion morphism, where $(E_V,
C^V)$ is the multiresolution of $(E,C)$ at $V$. Then
$$K_0(C^*(E_V,C^V))\cong K_0(C^*(E,C)) \oplus \mathbb Z^W$$
where $W$ is the set of all vertices $v(x_{t_1}^{(1)},\dots , x_{t_{k_u}}^{(k_u)})$,
where $u\in V$, $x^{(i)}_{t_i}\in X_i^{u}$ for all $i$, and $t_i>1$ ,  $t_j>1$
for at least two different indices $i$ and $j$.
We have
$$|W| = \sum _{u\in V} \Big(\prod _{i=1}^{k_u} |X^u_i|- \sum
_{i=1}^{k_u} |X^u_i| +k_u-1 \Big).$$
\end{lemma}

\begin{proof} For a commutative monoid $M$, we denote by $G(M)$ the
universal group of $M$. Given a monoid homomorphism $f\colon M_1\to
M_2$, there is an associated group homomorphism \linebreak$G(f)
\colon G(M_1)\to G(M_2)$. These assignments define a functor $G$
from the category of commutative monoids to the category of abelian
groups.

Note that \cite[Theorem 5.2]{AG2} implies that, for every finitely
separated graph $(E,C)$, the group $K_0(C^*(E,C))$ is isomorphic to
the universal group of $M(E,C)$. More precisely, we have that the
natural map $M(E,C)\to \mon{C^*(E,C)}$ induces a group isomorphism
$G(M(E,C))\cong G(\mon{C^*(E,C)}=K_0(C^*(E,C))$.

Set $\mu= M(\iota)$, where $M(\iota )\colon M(E,C)\to M(E_V,C^V)$ is
the natural map (see \cite{AG}). Note that, since $s(r^{-1}(V))\cap V=\emptyset$, $(E_V,C^V)$ can be
obtained as
the last term of a finite sequence of separated graphs, each one
obtained from the previous one by performing the multiresolution
process with respect to a single vertex, with no new arrows in
$r^{-1}(V)$ for all the graphs of the sequence.
We may thus suppose that
$V=\{v\}$ for a single vertex $v$ in $E^0$.

Set $C_v=\{ X_1,\dots ,X_k\}$, and write $X_i= \{
x^{(i)}_t\}_{t=1,\dots , |X_i|}$. Let $F$ be the free commutative monoid
on generators  $a(x^{(1)}_{t_1},x^{(2)}_{t_2},\dots ,
x^{(k)}_{t_k})$, for $t_i\in \{ 1,\dots ,|X_i| \}$, $i\in \{ 1,\dots
,k \}$. Let $M$ be the commutative monoid given by generators
$$\{ b(x^{(i)}_t)\mid t=1,\dots , |X_i|,\,\, i=1,\dots ,k \}$$
subject to the relations $\sum _{t=1}^{|X_i|}  b(x^{(i)}_t) =
\sum_{s=1}^{|X_j|} b(x^{(j)}_s)$, for $1\le i<j\le k$.

There is a unique monoid homomorphism $\eta: M\rightarrow M(E,C)$ sending
$b(x^{(i)}_t)$ to $[s(x^{(i)}_t)]$ for $1\le t\le |X_i|$, and there
is a unique homomorphism $\eta': F\rightarrow M(E_V,C^V)$ sending
$a(x^{(1)}_{t_1},\dots ,x^{(k)}_{t_k}) \mapsto
[v(x^{(1)}_{t_1},\dots ,x^{(k)}_{t_k})]$ for $1\le t_i\le |X_i|$,
$1\le i\le k$. There is  a commutative diagram as follows:
\begin{equation} \label{monres}
\xymatrixrowsep{3pc}\xymatrixcolsep{5pc} \xymatrix{
M \ar[r]^{\psi} \ar[d]_{\eta} & F \ar[d]^{\eta'} \\
M(E,C) \ar[r]^{\mu} &M(E_V,C^V) }
\end{equation}
where $\psi $ is given by the formula (\ref{eq:Gofpsi}) on the generators $b(x_{t_i}^{(i)})$
of $M$.
As noted in the proof of \cite[3.8]{AE}, an easy adaptation of the
proof of \cite[Lemma 8.6]{AG} gives that (\ref{monres}) is a pushout
in the category of commutative monoids. It is a simple matter to check
that the functor $G(-)$ transforms a pushout diagram in the category
of commutative monoids to a pushout diagram in the category of abelian
groups. Since $G(M(E,C))\cong K_0(C^*(E,C))$ and $G(M(E_V,
C^V))\cong K_0(C^*(E_V, C^V))$, we get a pushout diagram
\begin{equation} \label{groupres}
\xymatrixrowsep{3pc}\xymatrixcolsep{5pc} \xymatrix{
G(M) \ar[r]^{G(\psi)} \ar[d]_{G(\eta)} & G(F) \ar[d]^{G(\eta')} \\
K_0(C^*(E,C)) \ar[r]^{G(\mu)} & K_0(C^*(E_V,C^V))}
\end{equation}
By Lemma \ref{lem:refin-exactsequence}, the map $G(\psi)$ is
injective and we can write
$$G(F)=G(\psi)(G(M))\oplus H\, ,$$
where $H$ is a free abelian group of rank $\prod_{i=1}^k |X_i| -\sum
_{i=1}^k |X_i|+k-1$. It follows easily from the usual description of
pushouts in the category of abelian groups that
$$K_0(C^*(E_V, C^V)) \cong K_0(C^*(E,C)) \oplus H\, .$$
Indeed, we have that the mentioned pushout is computed as the
quotient group
$$(K_0(C^*(E,C))\oplus G(F))/T\, ,$$
where $T$ is the subgroup given be the elements of the form
$(G(\eta)(x), - G(\psi)(x))$, for $x\in G(M)$. It is quite easy to
check that $K_0(C^*(E,C))\oplus G(F) = (K_0(C^*(E,C))\oplus H)
\oplus T$, from which the result follows.
\end{proof}

\medskip

\begin{remark}
\label{rem:dual1} We may explicitly describe the Pontrjagin dual of
$K_0(C^*(E_V,C^V))$ using Lemma \ref{lem:K0multires}.

For any separated graph $(E,C)$, the Pontrjagin dual of the group
$K_0(C^*(E,C))$ can be thought of as the set of functions $\lambda
\colon E^0\to \mathbb T$ which are
invariant by the relations,
that is for every vertex $v\in E^0$ and every $X\in C_v$ we must
have
$$ \lambda (v)= \prod_{x\in X} \lambda (s(x))\, . $$
We get from Lemma \ref{lem:K0multires} that
$$\widehat{K_0(C^*(E_V,C^V))}\cong \widehat{K_0(C^*(E,C))} \oplus \mathbb T^W ,$$
that is, the character $\lambda \in \widehat{K_0(C^*(E_V,C^V))}$ is
determined by its values on the vertices of $E^0$ and on the
vertices $v(x^{(1)}_{t_1}(u),\dots ,x^{(k_u)}_{t_{k_u}}(u))$, for
$u\in V$, where $C_u=\{X^u_1,\dots ,X^u_{k_u}\}$,
$x^{(i)}_{t_i}(u)\in X^u_i$, and $t_i>1$, $t_j>1$  for at least two different $i$ and $j$.
We now indicate
how to determine the values of $\lambda$ at the remaining vertices
of $E_V$. Fix a vertex $u$ in $V$. To simplify notation, we will
suppress the dependence on $u$ in the notation. The elements of
$C_u$ will be denoted by $X_1,\dots ,X_k$. For each index $i$ and
every $t_i>1$, we have \begin{align*} \lambda (v(x^{(1)}_1, & \dots
, x^{(i-1)}_1,x^{(i)}_{t_i},x^{(i+1)}_1,\dots ,x^{(k)}_1)= \lambda
(s(x^{(i)}_{t_i}))\cdot \\
 & \Big[ \prod_{(s_1,\dots ,s_{i-1},s_{i+1},\dots ,s_k)\ne (1,1,\dots
,1)} \lambda (v(x^{(1)}_{s_1},\dots
,x^{(i-1)}_{s_{i-1}},x^{(i)}_{t_i},x^{(i+1)}_{s_{i+1}},\dots
x^{(k)}_{s_k}))\Big]^{-1}
\end{align*}
So all the values are determined except for $\lambda
(v(x^{(1)}_1,\dots , x^{(k)}_1))$. Since $[u]= \sum _{t_1,\dots
,t_k}[v(x^{(1)}_{t_1},\dots , x^{(k)}_{t_k})]$ in $K_0(C^*(E_V,
C^V))$, we must have
$$\lambda
(v(x^{(1)}_1,\dots , x^{(k)}_1))= \lambda (u)\cdot
\Big[\prod_{(t_1,\dots ,t_k)\ne (1,1,\dots ,1)}\lambda
(v(x^{(1)}_{t_1},\dots , x^{(k)}_{t_k})) \Big]^{-1}.$$
  This is how all of the values of the character
   $\lambda$ are determined from the given values.  \end{remark}

\section{$K_0$ for the tame C*-algebra of a finite bipartite separated graph}
\label{sect:bipsepgraphs}

In this section, we will obtain a description of $K_0(\mathcal O (E,C))$ for any finite bipartite
separated graph $(E,C)$. This will be used in Section \ref{sect:reduction-lemmas} to get a formula
for general finitely separated graphs.

We first recall some basic terminology and our graph construction
from \cite{AE}.

\begin{definition} (\cite{AE})
\label{def:bipartitesepgraph} Let $E$ be a directed graph. We say
that $E$ is a {\it bipartite directed graph} if $E^0= E^{0,0} \sqcup
E^{0,1}$, with all arrows in $E^1$ going from a vertex in $E^{0,1}$
to a vertex in $E^{0,0}$. To avoid trivial cases, we will always
assume that $r^{-1}(v)\ne \emptyset$ for all $v\in E^{0,0}$ and
$s^{-1}(v)\ne \emptyset $ for all $v\in E^{0,1}$.

A {\it bipartite separated graph} is a separated graph $(E,C)$ such
that the underlying directed graph $E$ is a bipartite directed
graph.
\end{definition}

\begin{construction} (\cite{AE})
\label{cons:complete-multiresolution} (a) Let $(E, C)$ be a finite
bipartite separated graph. We define a nested sequence of finite
separated graphs $(F_n, D^n)$ as follows. Set $(F_0,D^0)=(E, C)$.
Assume that a nested sequence
$$(F_0,D^0)\subset (F_1,D^1)\subset\dots \subset (F_n,D^n)$$
has been constructed in such a way that for $i=1,\dots ,n$, we have
$F_i^0=\bigsqcup _{j=0}^{i+1} F^{0,j}$ for some finite sets
$F^{0,j}$ and $F_i^1=\bigsqcup _{j=0}^i F^{1,,j}$, with $s(F^{1,j})=
F^{0,j+1}$ and $r(F^{1,j})= F^{0,j}$ for $j=1,\dots ,i$. We can
think of $(F_n, D^n)$ as a union of $n$
 bipartite separated graphs. Set $V_n= F^{0, n}$,
 and let $(F_{n+1}, D^{n+1})$ be the multiresolution of $(F_n, D^n)$
 at $V_n$. (Note that there are no edges between elements of $V_n$.)
  Then $F_{n+1}^0= F_n^0\bigsqcup F^{0,n+2}=\bigsqcup _{j=0}^{n+2} F^{0,j}$ and
 $F_{n+1}^1= F_n^1\bigsqcup F^{1,n+1}= \bigsqcup _{j=0}^{n+1}
 F^{1,j}$, with $s(F^{1,n+1}) = F^{0,n+2}$ and $r(F^{1,n+1})=s(F^{1,n})= F^{0,n+1}$.

\smallskip

\noindent (b) Let
 $$(F_{\infty}, D^{\infty})= \bigcup _{n=0}^{\infty} (F_n,
D^n) \, .$$
 Observe that $(F_{\infty}, D^{\infty})$ is the direct limit of the
 sequence $\{ (F_n,D^n) \}$ in the category $\FSGr$ defined in
 \cite[Definition 8.4]{AG}. We call
$(F_{\infty}, D^{\infty})$ the {\it complete multiresolution} of
$(E,C)$.

\smallskip

\noindent (c) We define a canonical sequence $ (E_n,C^n)$ of finite
bipartite separated graphs as follows:
\begin{enumerate}
\item Set $(E_0,C^0)= (E,C)$.
\item $E_n^{0,0}= F^{0,n}$,  $E_n^{0,1}=F^{0,n+1}$, and $E_n^1= F^{1,n}$.
Moreover $C^n_v= D^n_v$ for all $v\in E_n^{0,0}$ and
$C^n_v=\emptyset$ for all $v\in E_n^{0,1}$.
\end{enumerate}
We call the sequence $\{ (E_n,C^n)\}_{n\ge 0}$ the {\it canonical
sequence of bipartite separated graphs} associated to $(E,C)$.
\end{construction}

We will need the following Lemma, whose proof is contained in \cite[Lemma 4.5]{AE}.

\begin{lemma}
\label{lem:refinement} Let $(E,C)$ be a finite bipartite separated
graph, let $(E_n,C^n)$ be the canonical sequence of bipartite
separated graphs associated to $(E, C)$, and let $(F_{\infty},
D^{\infty})$ be the complete multiresolution of $(E,C)$. Then the
following properties hold:
\begin{enumerate}
\item[(a)]  For each $n\ge 0$, there is a natural isomorphism
$$\varphi _n \colon M(E_{n+1}, C^{n+1}) \longrightarrow M((E_n)_{V_n},
(C^n)^{V_n}), $$ where $V_n= E_n^{0,0}= F^{0,n}$.
\item[(b)] For each $n\ge 0$, there is a canonical embedding
$$\iota _n \colon
M(E_n, C^n)\to M(E_{n+1}, C^{n+1}).$$
\item[(c)] The canonical inclusion $j_n\colon (E_n,C^n)\to (F_n,
D^n)$ induces an isomorphism  $$M(j_n)\colon M(E_n,C^n)\to M(F_n,
D^n).$$
\item[(d)] We have $M(F_{\infty}, D^{\infty}) \cong  \varinjlim (M(E_n, C^n),\iota _n)$.
\end{enumerate}
\end{lemma}

\medskip

Let $(E,C)$ be a finite bipartite separated graph, with
$r(E^1)=E^{0,0}$ and $s(E^1)=E^{0,1}$. Let $\{ (E_n, C^n)\}_{n\ge
0}$ be the canonical sequence of bipartite separated graphs
associated to it (see Construction
\ref{cons:complete-multiresolution}(c)), and let $B_n$ be the
commutative C*-subalgebra of $C^*(E_n,C^n)$ generated by $E_n^0$.

\begin{theorem} {\rm (cf. \cite[Theorem 5.1]{AE})}
\label{thm:algebras} With the above notation, for each $n\ge 0$,
there exists a surjective homomorphism $$\phi _n \colon
C^*(E_n, C^n)\twoheadrightarrow C^*(E_{n+1}, C^{n+1}).$$
Moreover,
the following properties hold:
\begin{enumerate}
\item[(a)]
 $\ker (\phi _n)$
is the ideal $I_n$ of $C^*(E_n,C^n)$ generated by all the
commutators $[ee^*, ff^*]$, with $e,f\in E_n^1$, so that
$C^*(E_{n+1}, C^{n+1})\cong C^*(E_n, C^n))/I_n$.
\item[(b)] The restriction of $\phi_n$ to $B_n$ defines an injective homomorphism from $B_n$
into $B_{n+1}$.
\item[(c)] There is a commutative diagram
\begin{equation}
\label{eq:commu-daigram1}
 \begin{CD}
G(M(E_n, C^n)) @>G(\iota_n)>> G(M(E_{n+1}, C^{n+1}))\\
@V{\cong}VV @VV{\cong}V  \\
K_0(C^*(E_n, C^n))@>K_0(\phi _n)>> K_0(C^*(E_{n+1}, C^{n+1}))
\end{CD}\end{equation}
where the vertical maps are the canonical maps, which are
isomorphisms by \cite[Theorem 5.2]{AG2}.
\end{enumerate}
\end{theorem}

Since we shall use it later, we recall here the definition of the map $\phi_n$ appearing
in Theorem \ref{thm:algebras}(a) (see the proof of \cite[Theorem 5.1]{AE}). The map
$\phi_n$ is defined on vertices $u\in E_n^{0,0}$ by the formula
$$\phi_n (u)= \sum _{(x_1,\dots ,x_{k_u})\in \prod_{i=1}^{k_u} X^u_i} v(x_1,\dots
,x_{k_u}), $$ where $C_u=\{ X_1^u,\dots , X^u_{k_u} \}$, and by $\phi
_n(w)=w$ for all $w\in E_n^{0,1}$. For an arrow $x_i\in X^u_i$, we have
$$\phi _n(x_i)= \sum _{x_j\in X^u_j, j\ne i} (\alpha
^{x_i}(x_1,\dots , \widehat{x_i},\dots ,x_{k_u}))^* \, ,$$ where
$\alpha ^{x_i}(x_1,\dots , \widehat{x_i},\dots ,x_{k_u})=\alpha
^{x_i}(x_1,\dots , x_{i-1}, x_{i+1}, \dots  ,x_{k_u})$.

To simplify the notation, we will write $D_n= F^{0,n}=E_n^{0,0}$ for
all $n\ge 0$.

Note that, for $n\ge 2$ we have a surjective map $r_n\colon D_{n}\to
D_{n-2}$ given by $r_n (v(x_1,\dots x_{k_u}))= u$, where $u\in
D_{n-2}$ and $x_i\in X^u_i$, and where, as usual, $C^{n-2}_u=\{X^u_1,\dots
, X^u_{k_u} \}$. For $n=2m$, we thus obtain a surjective map $\gotr
_{2m}=r_2\circ r_4\circ \cdots \circ r_{2m}\colon D_{2m}\to D_0$.
Similarly, we have a map $\gotr_{2m+1}= r_3\circ r_5 \circ \cdots
\circ r_{2m+1} \colon D_{2m+1}\to D_1$. We call $\gotr (v)$ the {\it
root} of $v$. Observe that we have
\begin{equation}
\label{eq:Ddisjintunion} D_{2n}=\bigsqcup _{v\in D_0}
\gotr_{2n}^{-1}(v); \qquad D_{2n+1}=\bigsqcup _{v\in D_1}
\gotr_{2n+1}^{-1}(v)
\end{equation}

Set $A_n= C^*(E_n,C^n)$. Then $\mathcal O (E,C)=  \varinjlim A_n$, and
$B_{\infty}= \varinjlim B_n= C(\Omega (E,C))$. We have a commutative
diagram as follows:

$$\xymatrix@!=6.5pc{B_0 \ar[r] \ar[d] & B_1 \ar[r] \ar[d] & B_2 \ar[d]
\ar @{.>}[rr] & &  C(\Omega(E,C))\ar[d]\\
A_0 \ar[r]^{\phi _0}  & A_1 \ar[r]^{\phi_1}  & A_2 \ar @{.>}[rr] & &
\mathcal O (E,C)}
$$

All the maps $B_n\to B_{n+1}$ are injective and all the maps $A_n\to
A_{n+1}$ are surjective.

Let $\mathbb F$ be the free group on $E^1$.  There is a natural {\it partial
action} $\theta$ of $\mathbb F$ on $\Omega (E,C)$ so that $$\mathcal
O (E,C)\cong C(\Omega (E,C))\rtimes_{\theta^*}\mathbb F .$$ Moreover
$(\Omega (E,C), \theta )$ is the universal $(E,C)$-dynamical system
(see \cite{AE}). Let us recall the definition:

\begin{definition}
\label{def:univtopspaces} An {\it $(E,C)$-dynamical system} consists
of a compact Hausdorff space $\Omega$ with a family of clopen
subsets $\{ \Omega_v \}_{v\in E^{0}}$ such that
$$\Omega = \bigsqcup _{v\in E^{0}} \Omega_v , $$ and, for each $v\in E^{0,0}$,
a family of clopen subsets $\{ H_x \}_{x\in r^{-1}(v)}$ of
$\Omega_v$, such that
$$\Omega_v= \bigsqcup _{x\in X} H_x \qquad \text{ for all } X\in C_v , $$
together with a family of homeomorphisms
$$\theta_x\colon \Omega_{s(x)} \longrightarrow H_x $$
for all $x\in E^1$.

Given two $(E,C)$-dynamical systems $(\Omega, \theta)$, $(\Omega ',
\theta')$, there is an obvious definition of equivariant map
$f\colon (\Omega, \theta) \to (\Omega',\theta ')$, namely $f\colon
\Omega \to \Omega '$ is {\it equivariant} if $f(\Omega_w)\subseteq
\Omega'_w$ for all $w\in E^{0}$, $f(H_x)\subseteq H'_x$ for all
$x\in E^1$ and $f(\theta_x(y))= \theta'_x (f(y))$ for all $y\in
\Omega_{s(x)}$.

We say that an $(E,C)$-dynamical system $(\Omega ,\theta)$ is {\it
universal} in case there is a unique continuous equivariant map from
every $(E,C)$-dynamical system to $(\Omega ,\theta )$.
\end{definition}

We write $\Omega (E,C)=\bigsqcup _{v\in E^0}\Omega(E,C)_v$, $\Omega
(E,C)_v=\bigsqcup _{x\in X} H_x$ for all $X\in C_v$ $\, \, $($v\in
E^{0,0}$), and $\, \theta _x\colon \Omega(E,C)_{s(x)}\to H_x \, $
for the structural clopen sets and homeomorphisms of the universal
$(E,C)$-dynamical system.

We have
$$\varprojlim _i (D_{2i}, r_{2i}) = \Omega ^0:= \bigsqcup _{v\in
D_0} \Omega (E,C) _v, \qquad \varprojlim _i (D_{2i+1},
r_{2i+1})=\Omega^1:= \bigsqcup_{v\in D_1} \Omega (E,C)_v .$$ In the
following $\mathfrak r_{2k,\infty}\colon \Omega ^0 \to D_{2k}$ and
$\mathfrak r_{2k+1,\infty}\colon \Omega ^1\to D_{2k+1}$ will denote
the canonical projective limit surjections. The family $\{ \mathfrak
r_{k,\infty}^{-1} (v)\mid v\in D_k, k=0,1,2, \dots \}$ is a basis of
clopen sets for the topology of $\Omega (E,C)$.

\begin{theorem}
\label{thm:K0C(E,C)} Let $(E,C)$ be a finite bipartite separated
graph, and let $\pi \colon C^*(E,C)\to \mathcal O (E,C)$ be the natural projection map.
 Then $K_0(\pi)$ is a split monomorphism, and its cokernel $H$ is a free abelian group. Moreover, there are subsets $W_k \subset D_k$, for $k=2,3,\dots ,
$ such that $H\cong \bigoplus_{k=2}^{\infty} \Z^{W_k} $. In particular, we have
$$K_0(\mathcal O(E,C))\cong K_0(C^*(E,C))\oplus \Big(\bigoplus_{k=2}^{\infty} \Z^{W_k} \Big).$$
\end{theorem}

\begin{proof}
We have $K_0(\mathcal O (E,C))\cong \varinjlim_k K_0(C^*(E_k,C^k))$.

By Theorem \ref{thm:algebras}(c), it is enough to compute the limit
$\varinjlim (G(M(E_k,C^k)), G(\iota _k))$. Now the map $\iota
_k\colon M(E_k,C^k)\to M(E_{k+1}, C^{k+1})$ is the composition of
the canonical map $\iota _{V_k}\colon M(E_k,C^k)\to M((E_k)_{V_k},
(C^k)^{V_k})$ and the isomorphism $\varphi_k^{-1}\colon
M((E_k)_{V_k}, (C^k)^{V_k})\to
 M(E_{k+1},C^{k+1})$ (cf. \cite[Lemma 4.5]{AE}).

By (the proof of) Lemma \ref{lem:K0multires}, there are subsets $W_i$
of $D_{i}$, for $i=2,3,\dots, $ and isomorphisms
$$\gamma _i\colon G(M((E_i)_{V_i}, (C^i)^{V_i}))\overset{\cong}{\longrightarrow}  G(M(E_i, C^i)) \oplus
\Z^{W_{i+2}}\, ,$$ such that  $\gamma _i([v])= [v]$ for all $v\in
E_i^0$, $i=0,1,2,\dots  $.

We construct by induction a family of group isomorphisms
$$\theta _i\colon G(M(E_i,C^i))\to K_0(C^*(E,C))\oplus \Z^{W_2}\oplus \cdots \oplus
\Z^{W_{i+1}} $$ such that all the diagrams

\begin{equation}
\begin{CD}
G(M(E_i,C^i)) @>{\theta _i}>>
  K_0(C^*(E,C))\oplus \Z^{W_2}\oplus \cdots \oplus
\Z^{W_{i+1}} \\
@V{G(\iota _i)}VV  @V{j^{(i)}_1 }VV   \\
G(M(E_{i+1}, C^{i+1})) @>{\theta _{i+1}}>>  K_0(C^*(E,C))\oplus
\Z^{W_2}\oplus \cdots \oplus \Z^{W_{i+2}}
\end{CD}
\end{equation}
are commutative, where $j^{(i)}_1$ is the natural inclusion. The map
$\theta _0\colon G(M(E,C))\to K_0(C^*(E,C))$ is defined to be the
natural isomorphism. Assume that $\theta_0,\dots ,\theta _k$ have
been defined for some $k\ge 0$. Define the map
$$\widetilde{\gamma}_k \colon G(M((E_k)_{C_k},
(C^k)^{V_k}))\longrightarrow K_0(C^*(E,C))\oplus \Z^{W_2}\oplus
\cdots \oplus \Z^{W_{k+2}}$$ by $\widetilde{\gamma}_k= (\theta
_k\oplus \text{id}_{\Z^{W_{k+2}}})\circ \gamma _k$. Define
$\theta_{k+1}= \widetilde{\gamma}_k \circ G(\varphi_k)$. Then the
two squares in the following diagram are commutative:
\begin{equation}
\begin{CD}
G(M(E_k,C^k)) @>{\theta _k}>>
  K_0(C^*(E,C))\oplus \Z^{W_2}\oplus \cdots \oplus
\Z^{W_{k+1}} \\
@V{G(\iota _{V_k})}VV  @V{j^{(k)}_1 }VV   \\
G(M((E_k)_{V_k},(C^k)^{V_k})) @>{\widetilde{\gamma}_k}>>
K_0(C^*(E,C))\oplus  \Z^{W_2}\oplus \cdots \oplus \Z^{W_k+2} \\
@V{G(\varphi _k^{-1})}V{\cong}V  @V{=}VV   \\
 G(M(E_{k+1},
C^{k+1})) @>{\theta _{k+1}}>> K_0(C^*(E,C))\oplus \Z^{W_2}\oplus
\cdots \oplus \Z^{W_{k+2}}
\end{CD}
\end{equation}
Since $G(\iota_k)=G(\varphi_k^{-1})\circ G(\iota_{V_k})$, we have
completed the induction step.

We obtain
$$K_0(\mathcal O(E,C))\cong \varinjlim _k (G(M(E_k,C^k)), G(\iota_k))\cong K_0(C^*(E,C))\oplus \Big(\bigoplus_{k=2}^{\infty} \Z^{W_k} \Big)\, ,$$
as desired.
\end{proof}

\begin{remark}
\label{rem:dualofOEC} The Pontrjagin dual of $K_0(\mathcal O (E,C))$
can be identified with the set of $\mathbb T$-valued measures
defined on the field $\mathbb K$ of clopen subsets of $\Omega
(E,C)$, which are invariant under the action of $\mathbb F$. This is
exactly the dual of the type semigroup $S(\Omega (E,C),\mathbb F,
\mathbb K)$, considered in \cite[Section 7]{AE}.

Let $\mathcal B =\{ \mathfrak r _{k,\infty}^{-1} (v)\mid v\in W_k,
k=2,3,\dots  \}$. Then $\mathcal B$ is a family of clopen subsets of
$\Omega :=\Omega (E,C)$, which together with $K_0(C^*(E,C))$,
determine the Pontrjagin dual of $K_0(\mathcal O (E,C))$. Namely any
character $\lambda $ on $K_0(\mathcal O (E,C))$ is determined by its
values on the structural clopen sets $\Omega _v:= \Omega (E,C)_v$,
$v\in E^0$ (which have to fulfill
the relations $\lambda (\Omega _v)=
\prod _{x\in X} \lambda (\Omega _{s(x)})$ for every $v\in D_0$ and
every $X\in C_v$), and by the values $\lambda (U)$, for $U\in
\mathcal B$, which can be arbitrary complex numbers of modulus one.
The values of $\lambda $ on the other clopen sets of $\Omega$ are
determined inductively by the rules indicated in Remark
\ref{rem:dual1}.
\end{remark}

\begin{remark}
\label{rem:dense} With suitable conditions of connectedness, the
open set $\bigcup _{U\in \mathcal B} U$ is a dense subset of
$\Omega$ (where $\mathcal B$ is as in Remark \ref{rem:dualofOEC}).
For instance, we consider the separated graph $(E,C)=(E(m,n), C(m,n))$
appearing in \cite[Example 9.3]{AE} (see also \cite{AEK}), with $1<m\le
n$. We have $D_0=\{v_0\}$, $D_1=\{v_1\}$, and $C_{v_0}=\{X^{v_0},Y^{v_0} \}$,
with $|X^{v_0}|=n$ and $|Y^{v_0}|=
m$. Now, we consider the multiresolution of $(E,C)$, and we use the notation
introduced before. We get $|C_v|=2$ if $v\in
D_{2k}$ and $|C_v|=n+m$ if $v\in D_{2k+1}$. We have
 $C_{v_1}= \{X^{v_1}_1,\dots ,X^{v_1}_n, Y_1^{v_1},\dots,
Y_m^{v_1} \}$, with
$$|X_i^{v_1}|= m, \qquad |Y^{v_1}_j|= n  \quad (1\le i\le n, \, 1\le
j\le m).$$ One checks inductively that, for $v\in D_{2k}$,
$C_v=\{X^v,Y^v \}$, with $|X^v|=|X^{v'}|$ and $|Y^v|=|Y^{v'}|$ for
all $v,v'\in D_{2k}$, and that, for $v\in D_{2k+1}$,
$C_v=\{X_1^v,\dots X_n^v,Y^v_1,\dots , Y^v_m\}$, with
$$|X_i^v|= |X_{i'}^{v'}|,\qquad |Y_j^v|=|Y^{v'}_{j'}| ,\quad (1\le
i,i'\le n,\, 1\le j,j'\le m, \,  v,v'\in D_{2k+1} ).$$ Moreover, one has, for $v\in
D_{2k}$, $|X^v|= |X_1^{w}|^{n-1}|Y_1^w|^m$,  and $|Y^v|=
|X_1^w|^n|Y^w_1|^{m-1}$ where $w\in D_{2k-1}$, and, for $v\in
D_{2k+1}$, $|X^v_i|= |Y^w|$ and $|Y_j^v|= |X^w|$ for any $w\in
D_{2k}$. This clearly implies that $|C_v|\ge 2$, and $|X|\ge 2$ for
all $v\in D_k$ and for all $X\in C_v$. Using these inequalities, and
the constructions made in Lemmas \ref{lem:refin-exactsequence} and
\ref{lem:K0multires}, it follows that, if $v\in D_k$ for some $k$,
then $r_{k+2}^{-1}(v)\cap W_{k+2}\ne \emptyset$. Therefore
$$\mathfrak r_{k,
\infty}^{-1}(v)\cap \mathfrak r_{k+2}^{-1}(W_{k+2})\ne \emptyset.$$

Since the family $\{\mathfrak r_{k,\infty}^{-1}(v)\mid v\in D_k,\,\,
k=0,1,2\dots \}$ is a basis for the topology of $\Omega$, we see
that $\bigcup _{U\in \mathcal B} U$ is a dense open set of $\Omega$.
\end{remark}

\section{Partial unitaries in amalgamated free products}
\label{sect:partialunitaries}

In this section, we explicitly compute the image of certain partial unitary classes under the
K-theory map
\begin{equation}
\label{eq:thomsen-map}
K_1(A_1*_B A_2) \to K_0(B),
\end{equation}
  defined in \cite{Thomsen}, where $A_1*_B A_2$ is an amalgamated free product. This will be used in Section \ref{sect:compK1}
  to compute $K_1(\mathcal O (E,C))$ for any finite bipartite separated graph $(E,C)$.

Assume that $A_1,A_2,B$ are separable unital C*-algebras, with $B$ finite-dimensional, and that there are unital
embeddings $\iota_k\colon B\to A_k$, $k=1,2$.
Let $j_k\colon A_k\to A_1*_B A_2$ be the canonical maps.

  In our computations below we will use a special case of a main result by Thomsen, namely  \cite[Theorem 2.7]{Thomsen}.

\begin{theorem}
\label{thm:thomsen}  Let $B$, $A_1$, $A_2$ be separable
C*-algebras. Assume that $B$ is finite-dimensional.  Then there is a $6$-term
exact sequence:
\begin{equation}
\label{eq:6first}
\begin{CD}
K_0(B) @>{\;({i_1}_*,{i_2}_*)\;}>> K_0(A_1)\oplus K_0(A_2) @>{\;{j_1}_*-{j_2}_*\;}>> K_0(A_1 \,*_{B}\, A_2)\\
@AAA & & @VVV \\
K_1(A_1 \,*_{B}\, A_2) @<{\;{j_1}_*-{j_2}_*\;}<< K_1(A_1)\oplus K_1(A_2)
@<{\;({i_1}_*,{i_2}_*)\;}<< K_1(B)
\end{CD}
\end{equation}
\end{theorem}

We will need an elementary
  lemma,
  which is surely well known to specialists.

\def \twoBytwo #1#2#3#4{\begin {pmatrix} #1 & #2 \\ #3 & #4 \end {pmatrix}}

\begin{lemma}
  \label {lem:elementary-one}
  Given a unital C*-algebra $A$, and a short exact sequence of C*-algebras
  $$
  0\rightarrow J \rightarrow A \overset{\pi}{\rightarrow} B\rightarrow 0,
  $$
  let $u\in U_n(B)$, meaning the set of all unitary $n\times n$ matrices over $B$.
  Suppose that $v\in U_{2n}(A)$ is
  such that
  $$
  \pi (v) = \twoBytwo u 0 0 {\tilde u^*}
  $$
  and  $[\tilde u]_1=[u]_1$, in $K_1(B)$.
  Then
  $$
  \delta ([u]_1) = \Big[ v\twoBytwo 1000 v^*\Big]- \Big[ \twoBytwo 1000\Big],
  $$
  where   $\delta \colon K_1(B) \to K_0 (J)$ is the index map.
  \end{lemma}

\begin{proof}
  We should initially observe that, when $\tilde u = u$, then the conclusion of the
  lemma is essentially the definition of $\delta$.
  Set
  $$
  p= \twoBytwo 1000,
  $$
  and
let $z\in U_{2n}(A)$ be such that
  $$
  \pi(z) = \twoBytwo u00{u^*},
  $$
  so we have  by definition that
  $$
  \delta ([u]_1) = [zpz^*] -[p].
  \eqno{(\star)}
  $$

  By taking the direct sum of  all unitary matrices in sight with a big enough identity matrix, one may suppose that
there is a continuous path $u_t$ of unitaries such that $u_0=u$, and $u_1=\tilde u$.  Therefore, the unitary matrix
$u\tilde u^*$ lies in the connected component of $U_n(B)$, so there exists $x$ in  $U_n(A)$,
such that $\pi(x)=u\tilde u^*$.  Setting
  $$
  w = \twoBytwo 100x,
  $$
  we then have that
  $$
  \pi(zwv^*) =
  \twoBytwo u00{u^*} \twoBytwo 100{u\tilde u^*} \twoBytwo {u^*} 0 0 {\tilde u} =
  \twoBytwo 1001,
  $$
  from where it  follows that the element
  $$
  y:= zwv^*
  $$
  lies in $U_{2n} (\widetilde{J})$.   Working within $K_0(\widetilde J)$ we then have that
  $$
  [vpv^*]_0 = [y(vpv^*)y^*]_0 = [zwpw^*z^*]_0 =  [zpz^*]_0,
  $$
  so the conclusion follows immediately from $(\star)$.
  \end{proof}

We start with an easy case. This case provides the motivation for the more
  sophisticated result that we need later.

\begin{lemma}
 \label{lem:easy-case}
  With the above notation, let $x$ and $y$ be partial isometries, with $x\in M_{\infty}(A_1)$ and $y\in M_{\infty}(A_2)$, such that
$$xx^*= e = yy^* ,\qquad x^*x= f= y^*y , \qquad \text{with } e,f\in M_{\infty}(B) .$$
Then the image of the partial unitary class $[j_2(y)j_1(x)^*]$ under the homomorphism (\ref{eq:thomsen-map}) is precisely $[f]-[e]\in K_0 (B)$.
    \end{lemma}

\begin{proof}
     It suffices to deal with the case where $x\in A_1$ and $y\in A_2$.

 Denote by $C$ the mapping cone of the map $B\to A_1\oplus A_2$ sending $b$ to $(\iota_1(b), \iota_2(b))$, that is,
 $$C= \{ (b,g_1,g_2) : g_i\in C_0(0,1]\otimes A_i,\ i=1,2,\ b\in B,\ g_1(1)= \iota_1(b),\ g_2(b)=\iota_2 (b) \}.$$
  Let $G\colon C\to S(A_1 *_B A_2) = C_0(0,1)\otimes (A_1*_BA_2)$ be Germain's *-homomorphism, given by
$$G(b,g_1,g_2) (t) = \begin{cases}  j_1(g_1(2t)), & t\in (0,\frac{1}{2}]
\\ j_2(g_2(2-2t)), & t\in [\frac{1}{2} , 1) . \end{cases} $$
Then, by the proof of \cite[Theorem 2.7]{Thomsen}, we have the following commutative diagram:
\begin{equation}
\label{eq:commu-daigramK1}
 \begin{CD}
K_1(A_1*_B A_2)  @>>> K_0(B)\\
@V{\delta}V{\cong}V @AA{p_*}A  \\
K_0(S(A_1*_B A_2)) @>G_*^{-1}>{\cong }> K_0(C).
\end{CD}
\end{equation}
 Here $p\colon C\to B$ is the natural map, which sends $(b,g_1,g_2)$ to $b$, and
 $\delta$ is the index map for the six-term exact sequence of K-groups obtained from the short exact sequence
  \begin{equation}
  \label{eq:ConeExactSequence}
  0\rightarrow S(A_1*_BA_2) \rightarrow C(A_1*_BA_2) \rightarrow A_1*_BA_2\rightarrow 0,
  \end{equation}
  where $C(A_1*_BA_2)$ is the  cone of $A_1*_BA_2$.

 In view of diagram (\ref{eq:commu-daigramK1}), it will be sufficient to find $z\in K_0(C)$ such that $p_*(z)= [f]-[e]$
 and $\delta ([yx^*]) = G_*( z)$. In order to simplify notation, we will suppress the reference to the maps $\iota_k, j_k$
 in the rest of the proof.

 Write
 \begin{equation}
\label{eq:u-and-v}
 u= \begin{pmatrix} 1-e & x \\ x^* & 1-f \end{pmatrix}
    ,\quad\text{and}\quad
  v = \begin{pmatrix}
 1-e & y \\ y^* & 1-f \end{pmatrix}.
 \end{equation}
  It is easy to see that $u$ and $v$ are self adjoint unitary matrices over $A_1$ and
  $A_2$, respectively.
Set $Q:= \frac{1-u}{2}$ and $Q':= \frac{1-v}{2}$. Observe that $Q$ and $Q'$ are the
 spectral
projections corresponding to
  the eigenvalue $-1$ of
  $u$ and $v$ respectively. Put
  $$
  P_0:= \twoBytwo1000, \quad \text{and}\quad  P_1:= \twoBytwo{1-e}00f,
  $$
  which we view as $2\times 2$ matrices over $B$.
  Consider the paths of unitaries
\begin{equation}
 \label{eq:ut-and-vt}
 u_t: = (1-Q) + e^{\pi it} Q
    ,\quad\text{and}\quad
  v_t:= (1-Q')+ e^{- \pi it} Q'
\end{equation}
joining $I_2$ with $u$ and $v$ in $A_1$ and $A_2$, respectively.
Consider the projection
$$ D = (P_1, g_1 , g_2 )$$
in $M_2(\widetilde{C})$, where
$$g_1 (t) = u_tP_0u_t^*, \qquad g_2(t) = v_tP_0 v_t^* .$$
(Observe that $g_1(1)= g_2(1) = P_1$ because $x^*x= f= y^*y$.) Set $z= [D]- [P_0]\in K_0 (C)$.
(Note that $z\in K_0(C)$ because the image of $D$ through the canonical map $M_2(\widetilde{C})\to M_2 (\C)$ is the projection $\begin{pmatrix}
                                                                                                                     1 & 0 \\ 0 & 0
                                                                                                                    \end{pmatrix}\in M_2(\C)$.)

Note that
$$p_*(z)  = [P_1] - [P_0] = [f] -[e] .$$
It remains to show that $\delta ([yx^*]_1) = G_*(z)$.

Notice that  $yx^*$ is a partial isometry, with final projection
  $$
  yx^*xy^* = yfy^* = yy^*yy^* = e,
  $$
  and initial projection
  $$
  xy^*yx^* = xfx^* =   xx^*xx^* = e.
  $$
  So $yx^*$ is in fact a partial unitary, and hence $[yx^*]_1$ is defined to be the $K_1$-class of the unitary element
  $$
  U:= 1-e+yx^*.
  $$
  We will therefore show  that
  $$
  \delta ([yx^*]_1) =
  \delta ([U]_1) =
  G_*(z),
  $$
  by applying Lemma \ref{lem:elementary-one} to the exact sequence (\ref{eq:ConeExactSequence}).

  Letting
  $$
  w=\twoBytwo x{1-e}{1-f}{x^*},
  $$
  an easy computation shows that
  $$
  w^*\twoBytwo{1-e+yx^*}001 w = \twoBytwo{1-f+x^*y}001,
  $$
  so we see that the unitary element
  $$
  \tilde U:= 1-f+x^*y
  $$
  has the same $K_1$-class as $U$.   Observe moreover that
  $$
  vu = \twoBytwo {1-e+yx^*}00{1-f+y^*x} = \twoBytwo U00{\tilde U^*}.
  $$

  In order to apply Lemma \ref{lem:elementary-one} we thus need to find a lifting for the above matrix in the
unitization of the cone over $A_1*_BA_2$, namely a continuous path connecting the identity matrix to the above matrix.
  Such a path is not hard to find, it is enough to take
  $$
  \gamma _t =
  \begin{cases}  u_{2t}, & t\in [0,\frac{1}{2}] \\
  \vrule height 25pt width 0pt
  \big((1-Q')+ e^{(2t-1)\pi i}Q' \big) u , & t\in [\frac{1}{2} , 1]. \end{cases}
  $$
By Lemma \ref{lem:elementary-one}, we then have that
  $$
  \delta ([yx^*]_1)=\delta ([U]_1) = [\gamma _t P_0 \gamma _t^*] - [P_0]\, ,
  $$
but now observe that
$$\gamma _t P_0 \gamma _t ^* = \begin{cases}  g_1(2t), & t\in [0,\frac{1}{2}]
\\ g_2(2-2t)  , & t\in [\frac{1}{2} , 1] \, , \end{cases} $$
because, for $t\in [\frac{1}{2} ,1]$, we have
 \begin{align*}
 & \big( (1-Q')+ e^{(2t-1)\pi i}Q'\big) u  P_0 u^* \big((1-Q')+  e^{(1-2t)\pi i}Q' \big) \\
 & =  \big((1-Q')+ e^{(2t-1)\pi i}Q' \big) v  P_0 v^* \big((1-Q')+  e^{(1-2t)\pi i}Q' \big) \\
 & = \big((1-Q')+ e^{-(2-2t)\pi i} Q' \big)  P_0 \big((1-Q')+  e^{(2-2t)\pi i}Q' \big) \\
 & = v_{2-2t}P_0 v_{2-2t}^* = g_2(2-2t).
 \end{align*}
  This shows that $G_*(z)= \delta ([yx^*]_1)$, as desired.
\end{proof}

We now give the more technical statement that will be needed in the proof of Theorem \ref{thm:lambda-isomorphism}.

\begin{theorem}
 \label{thm:general-case}
 Let $e,f$ be projections in $M_{\infty} (B)$ and suppose that we have orthogonal decompositions
 $$e= e_1\oplus e_2= g_1\oplus g_2, \qquad  f= f_1\oplus f_2 = h_1\oplus h_2, $$
with $e_i,g_i,f_i,h_i\in M_{\infty} (B)$, for $i=1,2$. Assume that $x_1, y_1$ are partial isometries in
$M_{\infty}(A_1)$, and $x_2,y_2$ are partial isometries in $M_{\infty}(A_2)$ such that
$$e_i=x_ix_i^*, \quad f_i=x_i^*x_i, \qquad g_i= y_iy_i^*, \quad  h_i = y_i^*y_i, $$
for $i=1,2$. Set $x:=j_1(x_1)+j_2(x_2)$ and $y:= j_1(y_1)+j_2(y_2)$. Then the image of the partial unitary class $[yx^*]$ under the homomorphism
(\ref{eq:thomsen-map})
is precisely $$([f_1]-[e_1])-([h_1]-[g_1]) \in K_0 (B).$$
\end{theorem}

\begin{proof}
 The proof is similar to the one of Lemma \ref{lem:easy-case}, but we need to solve some technical complications.

 We will assume that $e,f\in B$, and so $x_i,y_i\in A_i$ as well. We will suppress any reference to the maps $\iota_k$ and $j_k$.
 As in the proof of Lemma \ref{lem:easy-case}, it suffices to find $z\in K_0(C)$ such that $p_*(z)= ([f_1]-[e_1])-([g_1]-[h_1])$
 and $\delta ([yx^*]_1)= G_*(z)$.

 Write
 \begin{equation}
\label{eq:u1-and-u2-and-U}
 u_1= \begin{pmatrix} 1-e_1 & x_1 \\ x_1^* & 1-f_1 \end{pmatrix},\quad u_2 = \begin{pmatrix}
 1-h_1 & y_1^* \\ y_1 & 1-g_1 \end{pmatrix} , \,  \, \text{ and }\,\, U= \begin{pmatrix} u_1 & 0 \\ 0 & u_2
 \end{pmatrix}.
 \end{equation}
Similarly, put
\begin{equation}
\label{eq:v1-and-v2-and-V}
 v_1= \begin{pmatrix} 1-g_2 & y_2 \\ y_2^* & 1-h_2 \end{pmatrix},\quad v_2 = \begin{pmatrix}
 1-f_2 & x_2^* \\ x_2 & 1-e_2 \end{pmatrix} , \,  \, \text{ and }\,\, V= \begin{pmatrix} v_1 & 0 \\ 0 & v_2
 \end{pmatrix}.
 \end{equation}
Note that $U$ is a self-adjoint unitary in $M_4(A_1)$ and $V$ is a self-adjoint unitary in $M_4(A_2)$. Consider the following projections in $M_4(B)$:
$$P_0= \text{diag} (1,0,1,0) ,\qquad P_1= UP_0U^* ,\qquad P_2= VP_0V^* .$$
Note that the projections $P_1$ and $P_2$ can be connected in $M_4(B)$. Indeed, consider
$$Z:= \begin{pmatrix}
       1-e & 0 & 0 & e \\
       0 & 1-f & f & 0 \\
       0 & f & 1-f & 0 \\
       e & 0 & 0 & 1-e
      \end{pmatrix}.$$
      Then $Z$ is a self-adjoint unitary in $M_4(B)$ and $ZP_1Z= P_2$. There exists a path $Z_t$ of  unitaries
      in $M_4(B)$ such that $Z_0= I_4$ and $Z_1= Z$.

 Set $Q:= \frac{1-U}{2}$ and $Q':= \frac{1-V}{2}$. Consider the paths of unitaries
\begin{equation}
 \label{eq:Ut-and-Vt}
 U_t: = (1-Q) + e^{\pi it} Q ,\qquad V_t:= (1-Q')+ e^{- \pi it} Q'
\end{equation}
joining $I_2$ with $U$ and $V$ in $A_1$ and $A_2$, respectively.
Consider the projection
$$ D = (P_2, g_1 , g_2 )$$
in $M_4(\widetilde{C})$, where
$$g_1 (t) = \begin{cases} U_{2t}P_0U_{2t}^* & t\in [0,\frac{1}{2} ] \\
                          Z_{2t-1} P_1Z_{2t-1}^* & t\in [\frac{1}{2}, 1] \end{cases},\qquad \qquad
 g_2(t) = V_tP_0 V_t^* .$$
(Observe that $g_1(1)= g_2(1) = P_2$.) Set $z= [D]- [P_0]\in K_0 (C)$.
Note that
$$p_*(z)  = [P_2] - [P_0] = [P_1] - [P_0] = ([f_1] -[e_1]) -([h_1]-[g_1]) .$$
It remains to show that $\delta ([yx^*]_1) = G_*(z)$.
A computation shows that
$$VZU = \begin{pmatrix}
        1-e & 0 & y & 0 \\
        0 & 1-f & 0 & y^* \\
        x^* & 0 & 1-f & 0 \\
        0 & x & 0 & 1-e
       \end{pmatrix} .$$
Observe that exchange of the second and third rows and columns of the matrix $VZU$ gives the unitary
$$W := \text{diag} \Big( \begin{pmatrix}
                1-e & y \\
                x^* & 1-f
               \end{pmatrix}, \, \begin{pmatrix}
              1-f & y^* \\
              x & 1-e
              \end{pmatrix} \Big)  \, , $$
              and the two unitaries appearing in this formula are equivalent to $1-e+yx^*$ and $1-e + xy^*$ respectively.
In particular, we have that $\Lambda (VZU) \Lambda = W$, where $\Lambda := \begin{pmatrix} 1 & 0 & 0 & 0 \\ 0 & 0 & 1 & 0 \\ 0 & 1 & 0 & 0 \\ 0 & 0 & 0 & 1 \end{pmatrix}$
is a self-adjoint unitary scalar matrix.  Therefore, Lemma \ref{lem:elementary-one} gives that
$$\delta ([yx^*]_1)=\delta ([1-e + yx^*]_1) = \delta (\Big[ \begin{pmatrix} 1-e & y \\ x^* & 1-f \end{pmatrix} \Big] _1) =  [\tilde{\gamma} _t \text{diag}(I_2, 0_2)  \tilde{\gamma} _t^*] - [\text{diag}(I_2, 0_2)]\, ,$$
where $\tilde{\gamma} _t $ is a unitary $4\times 4$ matrix over the unitization of the cone of $A_1*_B A_2$,
such that $\tilde{\gamma} _1 = W = \Lambda (VZU) \Lambda$.

Define $\gamma _t = \Lambda  \tilde{\gamma}_t \Lambda $. Then $\gamma _t$ is a unitary $4\times 4$ matrix such that $\gamma _1 = VZU$.
  Moreover, using the above computation, we get
\begin{align*}
 \delta ([yx^*]_1) & = [\Lambda \gamma_t (\Lambda\ \text{diag}(I_2, 0_2)  \Lambda ) \gamma_t^* \Lambda ] - [\text{diag}(I_2, 0_2)] \\
 & = [\Lambda \gamma_t  P_0 \gamma _t^* \Lambda] - [P_0] \\
 & = [\gamma _t P_0 \gamma _t ^*] - [P_0]
\end{align*}
in $K_0(S(A_1*_BA_2))$. In conclusion, we get that $ \delta ([yx^*]_1) = [\gamma _t P_0\gamma_t^*]-[P_0]$, where $\gamma _t$ is any unitary $4\times 4$ matrix over the unitization
of the cone of $A_1*_BA_2$ such that $\gamma_1= VZU$.

Consider the following unitary
$$\gamma _t = \begin{cases}  U_{4t}, & t\in [0,\frac{1}{4}]
\\ Z_{4t-1}U, & t\in [\frac{1}{4}, \frac{1}{2} ] \\
\Big[(1-Q')+ e^{(2t-1)\pi i}Q' \Big] ZU , & t\in [\frac{1}{2} , 1] . \end{cases} $$
Then $\gamma _0 = I_2$, and $\gamma _1 = VZU$, and so we can use $\gamma _t$ to compute $\delta ([yx^* ]_1)$.
  Just as in \ref{lem:easy-case}, we get
$G_*(z)= [\gamma_t P_0 \gamma _t^* ] - [P_0] = \delta ([yx^*]_1)$, as desired.
\end{proof}

\section{Computation of $K_1(\mathcal O (E, C))$.}
\label{sect:compK1}

In this section, we compute $K_1(\mathcal O (E,C))$ for any finite bipartite separated graph $(E,C)$.

\smallskip

Let $(E,C)$ be a finitely separated graph. For $v,w\in E^0$ and $X\in C_v$, denote by $a_X(w,v)$ the number of arrows in $X$ from
$w$ to $v$.

We denote by $1_C\colon
\Z^{(C)}\rightarrow \Z^{(E^0)}$ and $A_{(E,C)}\colon \Z^{(C)}\rightarrow \Z^{(E^0)}$ the
homomorphisms defined by
$$1_C(\delta_X)= \delta_v \qquad\text{and}\qquad A_{(E,C)}(\delta_X)= \sum _{w\in
E^0}a_X(w,v)\delta_w \qquad\quad (v\in E^0,\;X\in C_v ),$$
where $(\delta_X)_{X\in C}$ and $(\delta_v)_{v\in E^0}$ denote the canonical basis of $\Z^{(C)}$ and
$\Z^{(E^0)}$ respectively.

With this notation, the $K$-theory of $C^*(E,C)$ has formulas which
look very similar to the ones for the non-separated case (\cite[Theorem 3.2]{RaeSzy}):

\begin{theorem} \cite[Theorem 5.2]{AG2}
\label{thm:KTHsepgraph} Let $(E,C)$ be a finitely separated graph,
and adopt the notation above. Then the $K$-theory of $C^*(E,C)$ is
given as follows:
\begin{align}
K_0(C^*(E,C)) &\cong \coker \bigl( 1_C-A_{(E,C)}\colon \Z^{(C)}\longrightarrow \Z^{(E^0)} \bigr),  \label{eq:K0THCsep} \\
K_1(C^*(E,C)) &\cong \ker \bigl( 1_C-A_{(E,C)}\colon \Z^{(C)}\longrightarrow \Z^{(E^0)} \bigr).  \label{eq:K1THCsep}
\end{align}
\end{theorem}

Many well known results on the computation of $K$-theory groups for algebras related to graphs involve the {\it
transpose} of the adjacency matrix, contrary to our result above.  The appearance or not of the transpose in such
formulas is the consequence of two key choices of convention: the direction of the arrows in the graph (this has
undergone a major change in the literature in recent years mostly to make the {\it source} and {\it range} of the edges
in the graph match the {\it initial} and {\it final} spaces of the corresponding generating partial isometries) and the
definition of the adjacency matrix itself.  For example, the appearance of the transpose in \cite[Theorem 3.2]{RaeSzy} is
the consequence of choosing the old convention for the direction of arrows in the graph and choosing the indexing of the
adjacency matrix in such a way that the $(v,w)$ entry corresponds to the number of arrows from $v$ to $w$.  Each change
in these conventions has the effect of toggling the appearance/non-appearance of the transpose, such as in
\cite{Raeburn}, where the transpose remained due to both conventions being reversed.  Our convention for the direction
of arrows is the modern one, namely the same as in \cite{Raeburn}, while we kept the convention for indexing the
adjacency matrix from \cite{RaeSzy}, hence there is no transpose in the matrix $A_{(E,C)}$ in the above formulas.

  The isomorphism in (\ref{eq:K0THCsep}) is given explicitly in \cite{AG2},
but this is not the case for the isomorphism in (\ref{eq:K1THCsep}). We will obtain such an explicit isomorphism in this section.

Using this, we will show the following result:

\begin{theorem}
 \label{thm:K1computed} Let $(E,C)$ be a finite bipartite separated graph.
  The natural map $C^*(E, C)\to \mathcal O (E,C)$ induces an isomorphism
  $K_1(C^*(E,C))\to K_1(\mathcal O (E,C))$. Consequently,
  $$K_1(\mathcal O (E, C)) \cong \ker (1_C-A_{(E,C)}).$$
\end{theorem}

\medskip

To show this result, it is enough to prove that, for any finite bipartite separated graph $(E,C)$,
the natural map $\phi_0\colon C^*(E,C)\to C^*(E_1, C^1)$ induces an isomorphism
$$K_1(\phi_0)\colon K_1(C^*(E,C))\to K_1(C^*(E_1,C^1)),$$
where $(E_1,C^1)$ is the first of the infinite collection of separated graphs $(E_n, C^n)$ associated to $(E,C)$
  (see Construction \ref{cons:complete-multiresolution}(c)).

We start by fixing some notation. Let $(E,C)$ be a finite bipartite separated graph. We will denote
$(F,D):= (E_1, C^1)$.  Then $F^{0,0}= E^{0,1}$ and $F^{0,1} = \bigsqcup _{u\in E^{0,0}} F^{0,1}_u$, where
$F^{0,1}_u$ is the set of all vertices $v(x^u_1,\dots ,x^u_{k_u})$ for $x^u_i\in X^u_i$, $i=1,\dots , k_u$,
being $C_u= \{ X^u_1,\dots , X^u_{k_u} \}$,  for any $u\in E^{0,0}$.

For $w\in F^{0,0}= E^{0,1}$, the set $D_w$ can be identified with $s_E^{-1}(w)$ (see Definition \ref{multiresatsetofvs}), so that the set $D$ can be identified with
$E^1$:
$$\Z^D = \Z ^{E^1} =\bigoplus _{u\in E^{0,0}} \Z ^{|X^u_1|+\dots + |X^u_{k_u}|}.$$
We will denote by $\mathfrak D _u = \{ b(x^u_i)\mid x^u_i\in X^u_i, i=1,\dots , k_u \}$ a basis of
$\Z ^{|X^u_1|+\dots + |X^u_{k_u}|}$, so that $\mathfrak D =\bigsqcup _{u\in E^{0,0}} \mathfrak D _u$ is a basis of
$\Z ^D$.

\medskip

On the other hand, $\Z ^{F^0} = \Z^{F^{0,0}} \oplus \Z^{F^{0,1}}$. We consider a basis
$$\{ a(x^u_1,\dots , x^u_{k_u}) \mid x^u_i \in X^u_i, u\in E^{0,0} \} $$
for $\Z ^{F^{0,1}}$.

For $u\in E^{0,0}$ and $i=2,\dots , k_u$, set $\gamma ^u_i= \sum_{x^u_1\in X^u_1} b(x^u_1) - \sum _{x^u_i\in X^u _i} b(x^u_i)$.
Let $Z_1$ be the subgroup of $\Z^D$ generated by the elements $\gamma^u _i$, for $i=2,\dots ,k_u$, $u\in E^{0,0}$.
The map $\Psi \colon  \Z^D\to \Z^{F^{0,1}}$ given by
$$\Psi (b(x^u_i)) = \sum _{X_1^u\times \dots \times X_{i-1}^u\times X_{i+1}^u\times \dots \times X_{k_u}^u}
a(x^u_1,\dots x_{i-1}^u,x^u_i,x^u_{i+1},\dots , x^u_{k_u})  $$
is clearly related to the map $G(\psi )$ considered in Lemma \ref{lem:refin-exactsequence}.

By the proof of Lemma \ref{lem:refin-exactsequence}, we have that $Z_1= \ker (\Psi )$, and that $Z_1$ is a free subgroup of $\Z^D$ with free basis
given by the elements $\gamma^u _i$, for $i=2,\dots ,k_u$, $u\in E^{0,0}$. We have
$$\Z^D= Z_1\oplus Z_2,$$
and $\Psi $ induces an isomorphism from $Z_2$ onto its image.

\medskip

Observe that the map $\Psi $ can be identified with the map $A_{(F,D)}$.
We obtain:

\begin{lemma}
 \label{lem:keridentified}
 Let $(E,C)$ be a finite bipartite separated graph. With the above notation, we have
 $$\ker (1_D-A_{(F,D)}) = \ker (s_{Z_1})\, ,$$
 where $s_{Z_1}\colon Z_1\to \Z^{E^{0,1}}$ is the restriction to $Z_1$ of the map $s_{\Z^D}\colon \Z^D\to \Z^{E^{0,1}}$
 defined by $s_{\Z^D}(b(x)) = \delta _{s(x)}$ for $x\in E^1$.
\end{lemma}

\begin{proof}
 We have to compute the kernel of the map $1_D-A_{(F,D)}\colon \Z^D=Z_1\oplus Z_2  \to \Z^{F^0}= \Z^{F^{0,0}}\oplus \Z^{F^{0,1}}$.
 Note that $1_D$ takes its values on $\Z^{F^{0,0}}= \Z^{E^{0,1}}$ and can be identified with the ``source map'' $s_{\Z^D}$.
 On the other hand the map $A_{(F,D)}$ takes all its values on $\Z^{F^{0,1}}$ and can be identified with the map
 $\Psi$ described above. Since $Z_1= \ker (\Psi)$, the map $1_D- A_{(F,D)}$ decomposes as
 $$ \begin{pmatrix}
     s_{Z_1} & s_{Z_2} \\ 0 & - \Psi|_{Z_2}
    \end{pmatrix}\colon Z_1\oplus Z_2\longrightarrow \Z^{F^{0,0}}\oplus \Z^{F^{0,1}} , $$
    where $s_{Z_2}$ is the restriction of $s_{\Z^D}$ to $Z_2$. Since $\Psi |_{Z_2}$ is injective we obtain that
$\ker (1_D-A_{(F,D)})= \ker (s_{Z_1})$, as desired.
\end{proof}

\begin{lemma}
 \label{lem:twokeridnetified}
 Let $(E,C)$ be a finite bipartite separated graph. With the above notation, we have a natural isomorphism
 $$\ker (1_C-A_{(E,C)}) \overset{\Phi}{\cong}  \ker (1_D-A_{(F,D)}) .$$
  \end{lemma}

\begin{proof}
 By Lemma \ref{lem:keridentified}, it suffices to establish an isomorphism
 $$\Phi  \colon \ker (1_C-A_{(E,C)}) \longrightarrow \ker (s_{Z_1}).$$
 Recall that $Z_1= \bigoplus_{u\in E^{0,0}} \bigoplus _{i=2}^{k_u} \gamma ^u_i\Z$.

 \smallskip

For $x= \sum _{u\in E^{0,0}} \sum _{i=1}^{k_u} n_i^u \delta _{X^u_i}\in \ker (1_C-A_{(E,C)})$, define
\begin{equation}
\label{eq:defintion-of-Phi}
\Phi (x) = \sum _{u\in E^{0,0}} \sum _{i=2}^{k_u} n^u_i \gamma _i^u.
\end{equation}
 We show that this is well-defined, that is, that $\sum _{u\in E^{0,0}} \sum _{i=2}^{k_u} n^u_i \gamma _i^u\in \ker (s_{Z_1})$.
 Since $x$ belongs to $\ker (1_C-A_{(E,C)})$, we have $\sum _{i=1}^{k_u} n_i^u= 0$ for all $u\in E^{0,0}$.
We also have $$- \sum _{u\in E^{0,0}}\sum _{i=1}^{k_u}  n_i^u a_{X_i^u} (w,u) =0$$
  for all $w\in E^{0,1}$. Substituting $n^u_1$ by $-\sum _{i=2}^{k_u} n_i^u$ gives
$$\sum _{u\in E^{0,0}} \sum _{i=2}^{k_u} n_i^u\Big ( a_{X_1^u}(w,u) - a_{X_i^u}(w,u)\Big) = 0$$
for all $w\in E^{0,1}$, which in turn gives that $ \sum _{u\in E^{0,0}} \sum _{i=2}^{k_u} n^u_i \gamma _i^u\in \ker (s_{Z_1})$.
Clearly $\Phi $ is a group homomorphism. The map
$\Upsilon\colon \ker (s_{Z_1}) \to \ker (1_C-A_{(E,C)})$ defined by
$$\Upsilon (\sum _{u\in E^{0,0}}\sum _{i=2}^{k_u} n_i^u \gamma _i^u)= \sum _{u\in E^{0,0}}\sum _{i=1}^{k_u} n_i^u \delta _{X_i^u}\, ,$$
where $n_1^u:= -\sum _{i=2}^{k_u} n_i^u$ gives the inverse of $\Phi$, so we have showed that $\Phi $ is an isomorphism.
 \end{proof}

 Observe that Lemma \ref{lem:twokeridnetified} and \cite[Theorem 5.2]{AG2} already give an isomorphism $K_1(C^*(E,C))\cong K_1(C^*(E_1, C^1))$.
 (Recall that $(F,D)=(E_1,C^1)$.) However, we need the fact that the natural surjection $\phi_0\colon C^*(E,C)\to C^*(E_1,C^1)$ induces a $K_1$-isomorphism.
 In order to obtain this, we are going to describe now an explicit isomorphism $\lambda_{(E,C)}\colon \text{ker}(1_C-A_{(E,C)}) \to K_1(C^*(E,C))$ for any finite bipartite separated graph $(E,C)$.
 This is interesting in its own sake, since it enables us to compute specific elements in $K_1(C^*(E,C))$.

Let $(E,C)$ be a finite bipartite separated graph. Let
\begin{equation}
\label{eq:x-in-the-ker}
x= \sum _{u\in E^{0,0}} \sum _{i=1}^{k_u} n_i^u \delta _{X_i^u} - \sum _{u\in E^{0,0}} \sum_{j=1}^{k_u} m_j^u \delta _{X_j^u}
\end{equation}
be an element in the kernel of $1_C- A_{(E,C)}$, where $n_i^u, m_i^u $ are non-negative integers and $n_i^um_i^u= 0$ for all $u,i$.
This means exactly that
\begin{equation}
\label{eq:sums-ranges}
\sum _{i=1}^{k_u} n_i^u =\sum _{j=1}^{k_u} m_j^u \qquad (\forall u\in E^{0,0})
\end{equation}
and
\begin{equation}
\label{eq:sums-sources}
 \sum_{u\in E^{0,0}} \sum _{i=1}^{k_u}  n_i^u a_{X_i^u} (w,u) =   \sum_{u\in E^{0,0}} \sum _{j=1}^{k_u}  m_j^u a_{X_j^u} (w,u) \qquad (\forall w\in E^{0,1}) .
\end{equation}
Let us denote by $N_w$ the number appearing in (\ref{eq:sums-sources}), for $w\in E^{0,1}$.

We define a matrix $Z$, whose rows are labeled by the set
$$\mathcal R _1 = \bigsqcup_{u,i} \{ X_i^u \}\times [1,n_i^u] , $$
where $[1,n]:=\{ 1,\dots ,n \}$, and the indices range over all $u\in E^{0,0}$, $i\in [1, k_u]$, such that $n_i^u\ge 1$ and whose columns are labeled by the
set
$$\mathcal C _1 = \bigsqcup _{w, u,i} \{ X_i^u \} \times [1, n_i^u] \times \{ w \} \times [1, a_{X_i^u}(w,u)] , $$
where the indices range over all $w\in E^{0,1}$, $u\in E^{0,0}$, $i\in [1, k_u]$, such that $n_i^u\ge 1$ and $a_{X_i^u} (w,u)\ge 1$.
The matrix $Z$ is associated to ``the positive part'' $\sum _{u,i} n_i^u \delta _{X_i^u}$ of $x$.
Given $w\in E^{0,1}$, $u\in E^{0,0}$ and $X_i^u\in C_u$, choose an ordering $z_1,\dots , z_{a_{X_i^u}(w,u)}$ of the set of arrows from $X_i^u$ going from $w$ to $u$.
The matrix $Z$ is the unique $\mathcal R _1\times \mathcal C_1$-matrix such that, for each $w\in E^{0,1}$, $u\in E^{0,0}$ and $X_i^u\in C_u$, the column labeled
by $(X_i^u, t, w, s)$ has a unique nonzero entry, and this nonzero entry is precisely $z_s$ in row $(X_i^u, t)$. In other words, the only nonzero entries of the row
labeled $(X_i^u,t)$, for $t\in [1,n_i^u]$, are precisely the edges from $X_i^u$ and these are distributed in the columns corresponding to $(X_i^u,t)$, their source vertex $w$ and the
ordering fixed on the sets of arrows from $X_i^u$ going from $w$ to $u$. With this description, it is clear that
\begin{equation}
\label{eq:ZZStar-and-ZstarZ}
ZZ^* =   \bigoplus _{u\in E^{0,0}} (\sum _{i=1}^{k_u} n_i^u)  \cdot u , \qquad Z^*Z = \bigoplus_{w\in E^{0,1}} N_w \cdot w .
\end{equation}
Similarly, we may associate a $\mathcal R _2\times \mathcal C _2$-matrix  $T$ to the ``negative part'' $\sum _{u\in E^{0,0}} \sum_{j=1}^{k_u} m_j^u \delta _{X_j^u}$
of $x$. The rows and columns do not match exactly, but they match after we apply a bijection. More concretely, we fix two bijections
$$\sigma _1 \colon \mathcal R _1 \to \mathcal R _2 , \quad \text{and} \quad \sigma_2\colon \mathcal C _1 \to \mathcal C _2 $$
such that $\sigma _1$ restricts to a bijection from $\bigsqcup _{i} (\{ X_i^u \} \times [1, n_i^u]) $ onto $\bigsqcup _{j} (\{ X_j^u \} \times [1, m_j^u]) $
for all $u\in E^{0,0}$, and $\sigma_2$ restricts to a bijection from  $\bigsqcup _{u,i} (\{ X_i^u \} \times [1, n_i^u]\times \{ w \} \times [1, a_{X_i^u}(w,u)]) $
onto $\bigsqcup _{u,j} (\{ X_j^u \} \times [1, m_j^u]\times \{ w \} \times [1, a_{X_j^u}(w,u)])$ for all $w\in E^{0,1}$. Note that this is possible because of
(\ref{eq:sums-ranges}) and (\ref{eq:sums-sources}). Define a $\mathcal R_1\times \mathcal C_1$ matrix $\sigma (T)$ by
$$\sigma (T)_{r_1,c_1} = T_{\sigma _1(r_1), \sigma_2 (c_1)} \, , \qquad r_1\in \mathcal R _1, c_1\in \mathcal C_1. $$
Finally we define the map $\lambda_{(E,C)} \colon \text{ker} (1_C- A_{(E,C)}) \to K_1(C^*(E,C))$ by
$$\lambda_{(E,C)} (x) = [U_x]_1 , \quad \text{where} \quad U_x= Z\sigma (T)^* .$$
It is easily checked that this map does not depend on the choices of orderings that we have made, and of the specific bijections $\sigma_1$ and $\sigma_2$.
Similarly, we can use  $[\sigma^{-1}(Z)T^*]_1$ to define $[U_x]_1$.

\begin{proposition}
 \label{prop:commuting-diagram}
With the notation above, the following diagram
 \begin{equation}
\begin{CD}
{\rm ker} (1_C- A_{(E,C)})  @>{\lambda}_{(E,C)}>> K_1(C^*(E,C))\\
@V{\Phi}V{\cong}V  @VV{K_1(\phi_0)}V \\
{\rm ker} (1_D - A_{(F,D)}) @>>{\lambda_{(F,D)}}> K_1(C^*(F,D))
\end{CD}
\end{equation}
is commutative.
\end{proposition}

\begin{proof}  Recall that $(F,D):=(E_1,C^1)$.
 Let $x$ be an element in $\text{ker}(1_C-A_{(E,C)})$, written as in (\ref{eq:x-in-the-ker}).
 Note that $[U_x]_1=[V_x]_1$ in $K_1(C^*(E,C))$, where $U_x= Z\sigma (T)^*$ and $V_x= \sigma (T)^*Z$.
 We now compute the image of $\sigma (T)^*Z$ under the map $\phi_0$. Consider a nonzero entry of this matrix, corresponding to
 row $\sigma_2^{-1}(X_j^{u'}, r', w', t')$ and column $(X_i^u, r, w, t)$.
 The entry will be of the form $y^*z$ for some $y\in X_j^{u'}$ and some $z\in X_i^u$, with $s(y)= w'$ and $s(z)= w$.
 Since the (nonzero) entry $y$ must be at position $((X_j^{u'}, r'), (X^{u'}_j, r', w', t'))$ in the matrix $T$, and $z$ must be at position
$((X_i^u, r), (X_i^u, r, w,t))$ in the matrix $Z$, we must necessarily have
$\sigma _1(X_i^u, r)= (X_j^{u'}, r')$. In particular, by the choice of $\sigma_1$, we must have $u'=u$ and  thus
$$\sigma _1(X_i^u, r)= (X_j^{u}, r').$$
 Set $y=y_j^u$ and $z= z_i^u$.
 We have
\begin{align}
\label{eq:entry-for-TstarZ}
\phi_0 (y^*z) & = \Big( \sum_{y_l^u\in X_l^u, l\ne j} \alpha^{y_j^u} (y_1^u, \dots , \widehat{y_j^u},\dots , y_{k_u}^u) \Big) \Big( \sum _{z_k^u\in X_k^u, k\ne i} \alpha ^{z_i^u}(z_1^u,\dots , \widehat{z_i^u},
 \dots , z_{k_u}^u) ^*\Big)\\
\notag
& = \sum_{z_l^u\in X_l^u, l\notin \{ i,j \}}  \alpha^{y_j^u} (z_1^u, \dots , \widehat{y_j^u},\dots ,z_i^u, \dots ,  z_{k_u}^u)
 \alpha ^{z_i^u}(z_1^u,\dots , y_j^u, \dots ,  \widehat{z_i^u},
 \dots , z_{k_u}^u) ^* \, ,
\end{align}

Now we wish to compute the image of $x$ under $\Phi$, where $\Phi$ is the isomorphism defined in Lemma \ref{lem:twokeridnetified}.
Using (\ref{eq:defintion-of-Phi}), the definition of $\gamma _i^u$, and the identification of $b(x)$ with $\delta _{X(x)}$, for $x\in E^1$, we obtain
\begin{align*}
\Phi  (x) & = \sum_{u\in E^{0,0}} \sum_{i=2}^{k_u} n_i^u (\sum _{x_1^u\in X_1^u}  \delta_{X(x_1^u)} - \sum_{x_i^u\in X_i^u} \delta_{X(x_i^u)} )
- \sum _{u\in E^{0,0}} \sum _{j=2}^{k_u} m_j^u
(\sum _{x_1^u\in X_1^u}  \delta _{X(x_1^u)} - \sum_{x_j^u\in X_j^u} \delta_{X(x_j^u)} )\\
& =  \sum_{u\in E^{0,0}}(\sum_{i=2}^{k_u} n_i^u)(\sum_{x_1^u\in X_1^u} \delta_{X(x_1^u)} ) + \sum_{u\in E^{0,0}} \sum_{j=2}^{k_u} m_j^u (\sum _{x_j^u\in X_j^u} \delta_{X(x_j^u)} )\\
& - \Big( \sum _{u\in E^{0,0}} (\sum_{j=2}^{k_u}  m_j^u) (\sum_{x_1^u\in X_1^u} \delta _{X(x_1^u)} ) + \sum _{u\in E^{0,0}} \sum _{i=2}^{k_u} n_i^u (\sum_{x_i^u\in X_i^u} \delta_{X(x_i^u)})   \Big)
\end{align*}
From (\ref{eq:sums-ranges}), we have
$$\sum _{i=2}^{k_u} n_i^u - \sum _{j=2}^{k_u} m_j^u = m_1^u-n_1^u  \qquad (u\in E^{0,0}), $$
and so we get from the above
$$\Phi (x) = \sum _{u\in E^{0,0}} \sum_{j=1}^{k_u} m_j^u (\sum _{x_j^u\in X_j^u} \delta_{X(x_j^u)})- \sum _{u\in E^{0,0}}\sum_{i=1}^{k_u} n_i^u
(\sum _{x_i^u\in X_i^u} \delta_{X(x_i^u)}).$$
Let $Z_1$ and $T_1$  be the matrices corresponding to the ``positive part'' and the ``negative part'' of $\Phi (x)$, respectively.
We will compute $\widetilde{\sigma} (Z_1) T_1^*$, where $\widetilde{\sigma}= (\widetilde{\sigma}_1, \widetilde{\sigma}_2)$ is defined later.
We can identify the set $\mathcal R_2'$ of rows of $T_1$ with $\mathcal C_1$. Indeed the column labeled $(X_i^u, r,w,t)$ corresponds to the
row $(X(z_t), r)$, where $z_t$ is the $t$-th element in the list of elements from $X_i^u$ which have source $w$.
Similarly, we can identify the set $\mathcal R_1'$ of rows of $Z_1$ with $\mathcal C_2$

Note that, given elements $x_p^u\in X_p^u$, $p=1,\dots , k_u$,  there is only one arrow in $X(x_i^u)$ with source
$v(x^u_1,\dots  , v_{k_u}^u )$, namely $\alpha^{x_i^u} (x_1^u,\dots ,\widehat{x_i^u},\dots ,x_{k_u}^u)$. Therefore the labeling of the set
$\mathcal C_ 2'$ of columns of $T_1$ is given by
$$\mathcal C _2' = \bigsqcup  \{ X(x_i^u ) \}
\times [1, n_i^u] \times \{ v(x_1^u, \dots , x_i^u,\dots , x_{k_u}^u ) \} \, ,$$
where the union is extended to all $x_i^u\in X_i^u$ such that $n_i^u >0$ and to all choices of $(k_u-1)$-tuples
$(x_1^u, \dots x_{i-1}^u, x_{i+1}^u,\dots x_{k_u}^u)\in X_1^u \times \cdots \times X_{i-1}^u\times X_{i+1}^u\times \cdots \times X_{k_u}^u$.
There is only one nonzero entry in the column of $T_1$ labeled $(X(x_i^u), r, v(x_1^u, \dots , x_i^u,\dots , x_{k_u}^u))$, which is
$\alpha ^{x_i^u} (x_1^u,\dots ,\widehat{x}_i^u, \dots , x_{k_u} ^u)$ at row $(X(x_i^u), r)$.

The maps $\widetilde{\sigma}_i$, $i=1,2$, are defined as follows. The map $\widetilde{\sigma}_1\colon \mathcal R_2'\to \mathcal R _1'$  is
defined to be $\sigma _2$, with the identification of $\mathcal R_2'$
and $\mathcal R_1'$ with $\mathcal C_1 $ and $\mathcal C_2$ outlined above, respectively. To define $\widetilde{\sigma}_2 \colon \mathcal C_2'\to \mathcal C_1'$,
put
$$\widetilde{\sigma} _2( X(x_i^u), r, v(x_1^u,\dots , x_{k_u}^u)) = (X(x_j^u), r',  v(x_1^u,\dots , x_{k_u}^u)) ,$$
where $\sigma _1 (X_i^u, r) = (X_j^u, r')$ for $r'\in [1, m_j^u]$. That is, $x_j^u$ is determined as the unique element of $X_j^u$ which appears in the $k_u$-tuple
$(x_1^u,\dots , x_{k_u}^u)$.

Now, we wish to compute the $(\sigma_2^{-1} (X_j^u, r', w', t'), (X_i^u, r, w,t))$-entry of the matrix $ \widetilde{\sigma} (Z_1) T_1^*$, where $\sigma _1(X_i^u, r) = (X_j^u, r')$.
Recall from the beginning of the proof that the edge corresponding to  $(X_i^u, w,t)$ is denoted by $z= z^u_i$ and that the edge corresponding to $(X_j^u, w',t')$ is
denoted by $y=y_j^u$, so we are dealing with the
$(\widetilde{\sigma}_1^{-1} (X(y_j^u), r'), (X(z_i^u), r))$-entry of $ \widetilde{\sigma} (Z_1) T_1^*$.
Now consider a pair
$$\Big( (X(y_j^u), r'), (X(y_j^u), r', v(y_1^u,\dots , y_{k_u}^u))\Big),  \quad \Big( (X(z_i^u), r), (X(z_i^u), r, v(z_1^u,\dots , z_{k_u}^u)) \Big) $$ of positions
in the matrices $Z_1$ and $T_1$ respectively, giving rise to a nonzero contribution to the $(\widetilde{\sigma}_1^{-1} (X(y_j^u), r'), (X(z_i^u), r))$-entry
of $\widetilde{\sigma}(Z_1)T_1^*$. Then we must have
$$\widetilde{\sigma}_2(X(z_i^u), r, v(z_1^u,\dots , z_{k_u}^u)) = (X(y_j^u), r', v(y_1^u, \dots , y_{k_u}^u)) .$$
By the definition of $\widetilde{\sigma}_2$ and the fact that $\sigma_1( X_i^u, r)= (X_j^u, r')$, this happens if and only if
$y_l^u= z_l^u$ for all $l=1,\dots, k_u$. Hence, the contribution will be
$$ \alpha ^{y_j^u}(z_1^u,\dots , \widehat{y_j^u},\dots , z_i^u,\dots , z_{k_u}^u) \alpha^{z_i^u}(z_1^u,\dots , y_j^u, \dots \widehat{z_i^u}, \dots ,z_{k_u}^u)^*.$$
So, the $(\sigma_2^{-1} (X_j^u, r', w', t'), (X_i^u, r, w,t))$-entry of the matrix $ \widetilde{\sigma} (Z_1) T_1^*$
is
$$\sum _{z_l^u\in X_l^u, l\notin \{i,j \}} \alpha ^{y_j^u}(z_1,\dots , \widehat{y_j^u},\dots , z_i^u,\dots , z_{k_u}^u) \alpha^{z_i^u}(z_1^u,\dots , y_j^u, \dots \widehat{z_i^u}, \dots ,z_{k_u}^u)^* , $$
which is precisely (\ref{eq:entry-for-TstarZ}). Note that the argument we have just used gives that, if $\sigma_1(X_i^u,r)\ne (X_j^u,r')$ then the
$(\sigma_2^{-1} (X_j^u, r', w', t'), (X_i^u, r, w,t))$-entry of the matrix $ \widetilde{\sigma} (Z_1) T_1^*$ is $0$, for all $w,w',t,t'$.
Thus, we obtain that $\phi_0 (\sigma (T)^*Z) = \widetilde{\sigma}(Z_1)T_1^*$, and so
$$ \lambda _{(F,D)} (\Phi (x)) = [\widetilde{\sigma}(Z_1) T_1^*]_1 = K_1(\phi_0) ([\sigma (T)^* Z]_1) = K_1(\phi_0) (\lambda_{(E,C)} (x)),$$
as desired.
\end{proof}

We now proceed to show that the map $\lambda_{(E,C)}$ is an isomorphism. Note that this allows us to
explicitly compute generators for $K_1(C^*(E,C))$ (see below for some examples).

\begin{theorem}
\label{thm:lambda-isomorphism}
Let $(E,C)$ be a finite bipartite separated graph. Then the map
$$\lambda_{(E,C)}\colon \ker (1_C- A_{(E,C)}) \to K_1(C^*(E,C))$$
is a group  isomorphism.
\end{theorem}

\begin{proof}
Set $\lambda : = \lambda_{(E,C)}$. It is easy to check that $\lambda $ is a group homomorphism.

To show injectivity, suppose that $\lambda (x) = 0$, where
$$x=\sum _{X\in \mathcal C} n_X \delta _X - \sum _{Y\in \mathcal D} m_Y \delta _Y , $$
where $\mathcal C, \mathcal D \subseteq C$, with $\mathcal C \cap \mathcal D = \emptyset$, and $n_X, m_Y>0$ for all $X,Y$.
It will be convenient to use the notations ${\bf r}(X) =u$ and ${\bf s}(X) = \sum _{x\in X} s(x)$, for $X\in C_u$.

Choose a partition $C = C_1\sqcup C_2$ such that $\mathcal C\subseteq C_1$ and $\mathcal D \subseteq C_2$. Then we have
$$C^*(E,C) = C^*(E_1, C_1) *_B C^*(E_2, C_2) , $$
where $B= C(E^0)$, $E_1$ is the restriction of $E$ to $C_1$, that is $(E_1)^0 = E^0$ and $E^1 = \cup C_1$, and similalry $E_2$ is the restriction of $E$ to  $C_2$.
Now observe that from (\ref{eq:sums-ranges}), (\ref{eq:sums-sources}) and  (\ref{eq:ZZStar-and-ZstarZ}) we get
\begin{equation}
\label{eq:good-relations-forZZstar-etc}
ZZ^* = TT^*= \sigma (T)\sigma (T)^*= \bigoplus _{X\in \mathcal C} n_X \cdot {\bf r}(X), \quad Z^*Z = T^*T = \sigma (T)^* \sigma (T) = \bigoplus _{X\in \mathcal C} n_X \cdot {\bf s}(X),
\end{equation}
where $Z$ and $T$ are the matrices associated to the positive and negative parts of $x$ respectively. Let $\Delta \colon K_1(C^*(E_1, C_1) *_B C^*(E_2, C_2))\to K_0(B)$ be the homomorphism associated to this amalgamated free product, as in (\ref{eq:6first}).
By Lemma \ref{lem:easy-case} and (\ref{eq:good-relations-forZZstar-etc}), we get
$$\Delta ([Z\sigma (T)^*]_1) = [Z^*Z]- [ZZ^*] = \sum _{X\in \mathcal C} n_X [{\bf s}(X)] - \sum _{X\in \mathcal C} n_X [{\bf r}(X)] .$$
Since the graph is bipartite, there are no cancellations in this sum, and therefore, if $x\ne 0$, then $\mathcal C \ne \emptyset $ and so
$\Delta (\lambda (x)) = \Delta ([Z\sigma (T)^*]_1) \ne 0$, showing that $\lambda (x) \ne 0$.

Finally we show that $\lambda $ is surjective. First, we observe the naturality of the map $\lambda $: If $C'\subseteq C$ and $E'$ is the restriction of $E$ to $C'$, then the
following diagram
\begin{equation}
 \begin{CD}
 \label{CD:naturality-of-lambda}
\ker (1_{C'} - A_{(E',C')}) @>{\lambda_{(E',C')}}>> K_1(C^*(E',C')) \\
@V\iota VV  @VV\iota ' V \\
\ker (1_{C} - A_{(E,C)}) @>{\lambda_{(E,C)}}>> K_1(C^*(E,C)).
 \end{CD}
\end{equation}
is commutative. This is clear from the definition.  We assume by induction that for all $C'\subsetneq C$, we have
that $\lambda _{(E',C')}$ is an isomorphism. If there is $C'\subsetneq C$ such that $\iota ' (K_1(C^*(E',C')))= K_1(C^*(E,C))$, then by the
commutativity of (\ref{CD:naturality-of-lambda}), we get that $\lambda_{(E,C)}$ is surjective.
So we can assume that $\iota ' (K_1(C^*(E',C')))\subsetneq K_1(C^*(E,C))$ for all $C'\subsetneq C$.

Now let $C'$ be such that $C\setminus C' = \{X\}$, for $X\in C$. The proof of \cite[Theorem 5.2]{AG2} gives that
$$K_1(C^*(E,C)) = K_1(C^*(E',C')) \oplus H \, ,$$
where $H$ is a cyclic group (see formula (5.9) in  \cite{AG2} and the comments below it). It is enough to show that the generator $v$ of $H$ belongs to the image of $\lambda_{(E,C)}$.
Let
$$\Psi \colon K_1(C^*(E,C)) \longrightarrow K_0 (B)$$
be the connecting map corresponding to the decomposition
$$C^*(E,C) = C^*(E',C')*_B C^*(E_{\{ X \}}, \{X \})$$
of $C^*(E,C)$ as an amalgamated free product, as in (\ref{eq:6first}).

Following the  notation in the proof of \cite[Theorem 5.2]{AG2}, set $A:= 1_{C'}- A_{(E',C')}$ and $B:= 1_{\{X\}} - A_{(E_{\{ X\}}, \{ X \} )}$. It is shown there that the map $\Psi$ restricts
to an isomorphism between $H$ and $\Psi (H)$, which is an infinite cyclic group. (Note that $H\ne 0$ by our assumption.) Moreover,
$$\Psi (H) = A(\Z^{C'}) \cap B(\Z \delta_{X}) .$$
Let $b= \Psi (v)$ be the generator of $\Psi (H)$. It suffices to find an element $g$ in the image of $\lambda_{(E,C)}$ such that $\Psi (g)= b$.
Now write
$$b= B(n_X\delta _X) = A(\sum_{Y\in C'} \lambda _Y \delta_Y) ,$$
where $n_X,\lambda _Y \in \Z$. We may assume that $n_X>0$. Now we consider the element $[Z\sigma (T)^*]_1\in K_1(C^*(E,C))$ associated to the element
$$x : = n_X\delta _X - \sum_{Y\in C'} \lambda_Y \delta _Y \in \ker (1_C- A_{(E,C)}) .$$
Then, with $A_1= C^*(E',C')$ and $A_2= C^*(E_{\{ X \}}, \{X \})$, we can decompose $Z=Z_1\oplus Z_2$ with $Z_1$ corresponding to the positive part
of $ - \sum_{Y\in C'} \lambda_Y \delta _Y $ and $Z_2$ corresponding to $n_X\delta _X$. There is no contribution of $A_2$ to the negative part of $x$,
so $T= T_1\oplus 0$, where $T_1$ corresponds to the negative part of $ - \sum_{Y\in C'} \lambda_Y \delta _Y $. We have
$$e:= Z_1Z_1^* + Z_2Z_2^* = T_1T_1^*,\qquad f:= Z_1^*Z_1+ Z_2^*Z_2 = T_1^*T_1 .$$
Therefore, by Theorem \ref{thm:general-case} and (\ref{eq:good-relations-forZZstar-etc}), we get
\begin{align*}
\Psi & ([Z\sigma (T)^*]_1)  = ([Z_1^*Z_1]-[Z_1Z_1^*])-([T_1^*T_1]-[T_1T_1^*])\\ =
& ([e]-[Z_1Z_1^*])-([f]-[Z_1^*Z_1])= [Z_2Z_2^*]-[Z_2^*Z_2]\\
& = n_X[{\bf r}(X)]- n_X [{\bf s}(X)] = B(n_X\delta_X) = b.
\end{align*}
This shows that $b = \Psi (\lambda_{(E,C)} (x))$, as wanted. The proof is complete.
\end{proof}

We can now obtain a proof of an enhanced version of the main result of this section (Theorem \ref{thm:K1computed}).

\begin{theorem}
 \label{thm:main-K1-first} Let $(E,C)$ be a finite bipartite separated graph, and let $\pi \colon C^*(E,C)\to \mathcal O (E,C)$ be the natural projection map.
 Then $\pi $ induces an isomorphism
 $$\pi_* \colon K_1(C^*(E,C)) \overset{\cong}{\longrightarrow} K_1(\mathcal O (E,C)) .$$
Moreover, the map $\pi_*\circ \lambda_{(E,C)}\colon \ker (1_C-A_{(E,C)}) \to K_1 (\mathcal O (E,C))$ is an isomorphism.
 \end{theorem}

\begin{proof} It follows from Lemma \ref{lem:twokeridnetified}, Theorem \ref{thm:lambda-isomorphism}, and Proposition \ref{prop:commuting-diagram} that
all the maps $K_1(\phi_n) \colon K_1(C^*(E_n,C^n)) \to K_1(C^*(E_{n+1}, C^{n+1}))$ are isomorphisms. Since $K_1(\mathcal O (E,C)) \cong
\varinjlim _n K_1(C^*(E_n, C^n))$, with $K_1(\phi _n)$ as the connecting maps, the result follows.

The last statement follows from the first and Theorem \ref{thm:lambda-isomorphism}.
\end{proof}

Another possible method to compute the $K$-groups of
  $\mathcal {O}(E,C)$
  is by realizing it as a partial crossed product,   and then using McClanahan's
  generalized Pimsner-Voiculscu exact sequence for crossed products by semi-saturated   partial actions of free groups
  \cite[Theorem 6.2]{McCla4}.

However the known groups in the above mentioned exact sequence turn out to be  quite large and difficult to manage, making a concrete calculation rather
difficult.  Nevertheless, after having computed $K_*(\mathcal {O}(E,C))$ by the methods employed in the present article,
we may use McClanahan's result to obtain the $K$-groups for the reduced version of $\mathcal {O}(E,C)$, which we will
now briefly discuss.

Recall from Section \ref{sect:bipsepgraphs} that $\mathcal O (E,C)$ is isomorphic to the \emph{full} crossed product
  $$C(\Omega (E,C))\rtimes_{\theta^*}\mathbb F,$$
  where $(\Omega (E,C), \theta )$ is the universal $(E,C)$-dynamical system.  The \emph{reduced} version of $\mathcal O
(E,C)$ may then be defined as follows:

\begin{definition}
  We shall denote by $\mathcal {O}_{red}(E,C)$ the \emph{reduced} crossed product
  $$C(\Omega (E,C))\rtimes_{\theta^*\!,\,red}\mathbb F.$$
  \end{definition}

\begin{corollary}
\label{cor:reduced-version}
  The natural map
  $$
  \lambda: \mathcal {O}(E,C) \to   \mathcal {O}_{red}(E,C)
  $$
  induces an isomorphism on $K$-groups.
\end{corollary}

\begin{proof}  It is enough to notice that the arrow marked $\lambda_*$ in   \cite[Theorem 6.2]{McCla4}
is an isomorphism by the Five Lemma.
 \end{proof}

\begin{example}
\label{exam:K1Omn}
The algebra $U_n^{{\rm nc}}$ is the C*-algebra generated by the entries of a universal $n\times n$ unitary matrix $U= [u_{ij}]$, see \cite{McCla1}.
This was generalized in  \cite{McCla2}, where the C*-algebra $U_{m,n}^{{\rm nc}}$ generated by a $m\times n$ unitary matrix was considered.
The K-theory of $U_n^{{\rm nc}}$ was found in \cite[Corollary 2.4]{McCla1}. The K-theory of $U_{m,n}^{{\rm nc}}$ was computed in \cite{AG2}, as a consequence of the
computation of the K-theory of C*-algebras of separated graphs, thus solving a conjecture raised by McClanahan in \cite{McCla2}.
Recall from \cite[Example 4.5]{AG2} that
$$C^*(E(m,n), C(m,n))\cong M_{n+1}(U_{m,n}^{{\rm nc}})\cong M_{m+1}(U_{m,n}^{{\rm nc}}). $$
We now get from Theorem \ref{thm:main-K1-first} and \cite[Theorem 5.2]{AG2}:
$$K_1(\mathcal O_{m,n}^{{\rm red}}) \cong K_1( \mathcal O _{m,n}) \cong K_1(U^{\text{nc}}_{m,n}) \cong \ker\left( \begin{pmatrix}
1 & 1 \\ -n & -m
\end{pmatrix} : \Z^2\to \Z^2 \right) \cong \begin{cases} \Z & \text{if } n=m\\
0 & \text{if } n>m
\end{cases} .$$
For $m=n$, setting $E:= E(m,n)$ and $C:= C(m,n)$, we recover the fact that $K_1(U_n^{{\rm nc}})$ is generated by the class of $U= (u_{ij})$.
Indeed, Theorem \ref{thm:main-K1-first} says that $K_1(C^*(E,C))$ is generated by $\lambda_{(E,C)} (x)$, where $x= \delta_X - \delta _Y$.
Now $\lambda_{(E,C)} (\delta _X -\delta _Y) = [ZT^*]_1$, with
$$Z= \begin{pmatrix} \alpha_1 & \cdots  & \alpha_n \end{pmatrix}, \qquad T = \begin{pmatrix} \beta_1 & \cdots & \beta_n \end{pmatrix}.$$
Thus $K_1(C^*(E,C))$ is generated by the class of the unitary $\sum_{i=1} ^n \alpha_i \beta_i ^*$ of $vC^*(E,C)v$.
The unitary $T^*Z = (\beta_i^* \alpha_j)$ in $M_n(wC^*(E,C)w)$ represents the same element and corresponds to $(u_{ij})$ under the canonical isomorphism
$wC^*(E,C)w \cong U_n^{{\rm nc}}$ (see \cite[Example 4.5]{AG2} and \cite[Proposition 2.12(1)]{AG}).  The images of these unitaries through the canonical projection
maps $C^*(E,C) \to \mathcal O _{n,n}\to \mathcal O _{n,n}^{{\rm red}}$ provide the generators of $K_1$ of these C*-algebras.
\end{example}

\begin{example}
 \label{exam:K1lamplight}
 We now consider, for $p\ge 2$, the bipartite separated graph $(E,C)$ with $p+1$ vertices $E^{0,0}= \{v\}$, $E^{0,1}= \{ w_1,\dots , w_p \}$ and $2p$ edges
 $E^1= \{ \alpha _1,\dots , \alpha_p, \beta_1,\dots , \beta_p \}$, with $s(\alpha_i) = s(\beta _i) = w_i$ and $r(\alpha_i) = r(\beta_i) = v$ for $i=1,\dots ,p$, and
 with $C= \{X,Y \}$, $X= \{\alpha _i \}$, $Y= \{ \beta _i \}$.
 It was observed in \cite[Lemma 5.5(2)]{Aone-rel} that $vC^*(E,C)v \cong C^*((\bgast _ {\Z} \Z _p )\rtimes \Z)$ and in \cite[Example 9.7]{AE} that
 $v\mathcal O (E,C) v \cong C^*(\Z _p \wr \Z)$, where $\Z_p \wr \Z= (\oplus_{\Z} \Z_p)\rtimes \Z$ is the wreath product of $\Z_p$ by $\Z$.
 (The latter groups are called the {\it lamplighter groups}.)
 Here we have
 $$1_C-A_{(E,C)} = \begin{pmatrix}
                    1 & 1 \\
                    -1 & -1 \\
                    \cdots & \cdots  \\
                    -1 & -1
                   \end{pmatrix}$$
and so, by using a similar computation as in Example \ref{exam:K1Omn}, we get that $K_1(C^*(E,C))$ is a cyclic group generated by $[\sum _{i=1}^n \beta_i \alpha_i^*]_1$,
where $u: = \sum_{i=1}^n \beta_i \alpha_i^*$ is a unitary in $v C^*(E,C)v$.
Observe that $u$ is the unitary corresponding to the generator of the copy of $\Z$ in $C^*((\bgast _ {\Z} \Z _p )\rtimes \Z)$ under the canonical isomorphism
between $vC^*(E,C)v$ and $C^*((\bgast _ {\Z} \Z _p )\rtimes \Z)$. (Only the  case $p=2$ was considered in \cite[Example 5.5(2)]{Aone-rel}, but the case where $p>2$ is completely analogous.)

Similarly we obtain that $K_1(C^*(\Z _p \wr \Z ))$ is generated by the class of the unitary in $C^*(\Z_p \wr \Z)$
corresponding to the generator of $\Z$.
  \end{example}

\section{Finitely separated graphs}
 \label{sect:reduction-lemmas}

 In this section we develop some methods which allow us to extend our results for finite bipartite separated graphs to general finitely separated graphs.
The methods combine the direct limit technology of \cite{AG} and \cite[Proposition 9.1]{AE}.

 \begin{theorem}
 \label{thm:mainKifinitegrs}
 Let $(E,C)$ be a finite separated graph. Then we have
 \begin{enumerate}
  \item The canonical map $\pi_{(E,C)}\colon C^*(E,C)\to \mathcal O (E,C)$ induces an injective split homomorphism
  $K_0(\pi) \colon K_0(C^*(E,C)) \to K_0 (\mathcal O (E,C))$. Moreover
   $$K_0(\mathcal O (E,C)) \cong K_0(C^*(E,C))\oplus H \cong \coker (1_C-A_{(E,C)})\oplus H, $$
  where $H$ is a free abelian group.
  \item The map $K_1(\pi_{(E,C)})\colon K_1( C^* (E,C)) \to  K_1(\mathcal O (E,C)) $ is an isomorphism.
   \end{enumerate}
\end{theorem}

   \begin{proof}
       For each separated graph $(E,C)$ there is a canonical finite bipartite separated graph $(\tilde{E}, \tilde{C})$ such that
    the following diagram is commutative
\begin{equation}
\label{eq:comdigram-for-the-double}
 \begin{CD}
M_2(C^*(E,C)) @>\cong>> C^*(\tilde{E}, \tilde{C})\\
@V{M_2(\pi_{(E,C)})}VV @VV{\pi_{(\tilde{E}, \tilde{C})}}V  \\
M_2(\mathcal O (E, C)) @>\cong>> \mathcal O (\tilde{E}, \tilde{C})
\end{CD}\end{equation}
where the horizontal maps are isomorphisms \cite[Proposition 9.1]{AE}.
    Apply $K_i$, $i=0,1$, to this diagram and use Theorems \ref{thm:K0C(E,C)} and \ref{thm:main-K1-first}.
    \end{proof}

 Now we start the preparations to obtain the results for finitely separated graphs.

 We first view the assignment $(E,C)\mapsto \mathcal O (E,C)$ as a functor on a certain category. We will only consider finitely separated graphs in this paper.
 We believe that suitable generalizations should be possible for general separated graphs. The category $\FSGr$ of finitely separated graphs was defined in \cite[Definition 8.4]{AG}.
 The objects of $\FSGr$ are
 all the finitely separated graphs. If $(E,C)$ and  $(F,D)$ are finitely separated graphs, then a morphism $\phi$ from $(E,C)$ to $(F,D)$
 is a graph homomorphism $\phi = (\phi^0, \phi^1):(E^0,E^1)\to (F^0, F^1)$ from $E$ to $F$ such that $\phi ^0$ is injective, and such that, for each $X\in C$ there is (a unique) $Y\in D$ such that
 $\phi ^1$ induces a bijection from $X$ onto $Y$.

 Given an object $(E,C)$ of $\FSGr$, a {\it complete subobject} of $(E,C)$ is a finitely separated graph $(F,D)$ such that $F$ is a subgraph of $E$ and $D$ is a subset of $C$.
 (In particular the edges of $F$ are exactly all the edges of $E$ which belong to some of the elements of the subset $D$ of $C$, i.e., $F^1= \sqcup_{Y\in D} Y$.) Note that a complete subobject
 corresponds essentially
 to the categorical notion of a subobject in the category $\FSGr$. By \cite[Proposition 1.6]{AG2}, $\FSGr$ is a category with direct limits, and $(E,C)\mapsto C^*(E,C)$ defines a continuous functor from
 $\FSGr$ to the category $\Cstaralg$ of C*-algebras. If $\phi$ is a morphism from $(E,C)$ to $(F,D)$, then the associated *-homomorphism $C^*(\phi)\colon C^*(E,C)\to C^*(F,D)$ is given
 by $C^*(\phi)(v)= \phi ^0(v)$ and $C^*(\phi)(e)= \phi^1 (e)$, for $v\in E^0$ and $e\in E^1$.

Let $(E,C)$ is a finitely separated graph. Define a partial order on the set of complete subobjects of $(E,C)$ by setting $(F,D)\le (F',D')$ if and only if $(F,D)$ is a complete subobject of $(F',D')$.

 \begin{proposition}
  \label{Ofunctor}
  The assignment $(E,C)\mapsto \mathcal O (E,C)$ defines a continuous functor from the category $\FSGr$ of finitely separated graphs to the category of C*-algebras.
  Moreover, for any finitely separated graph $(E,C)$, we have $\mathcal O (E,C)= \varinjlim \mathcal O (F,D)$ where the limit is over the directed set of all the
 finite complete subobjects of $(E,C)$.
   \end{proposition}

\begin{proof}
 The second part follows from the first and the fact that every object in $\FSGr$ is the direct limit of the directed family of its finite complete subobjects (\cite[8.4]{AG}).

 For a finitely separated graph $(E,C)$, denote by $J_{(E,C)}$ the closed ideal of $C^*(E,C)$ generated by all the commutators $[e(u), e(u')]$, where $u,u'$ belong to the
 multiplicative subsemigroup of $C^*(E,C)$ generated by $E^1\cup (E^1)^*$. By definition, we have $\mathcal O (E,C)= C^*(E,C)/J_{(E,C)}$.

 If $\phi$ is a morphism from $(E,C)$ to $(F,D)$ in $\FSGr$, then $\phi $ induces a *-homomorphism $C^*(\phi)\colon C^*(E,C)\to C^*(F,D)$. Clearly, we have
 $C^*(\phi) (J_{(E,C)}) \subseteq J_{(F,D)}$, so that there is an induced map $\mathcal O (\phi) \colon \mathcal O (E,C)\to \mathcal O (F,D)$, and we obtain a functor $\mathcal O$ from
 $\FSGr$ to $\Cstaralg$. To show that this functor is continuous, let $\{ (E_i, C^i),\varphi _{ji}, i\le j, i,j\in I \}$ be a directed system in the category $\FSGr$.
 By \cite[Proposition 1.6]{AG2}, we have $C^*(E,C) = \varinjlim _{i\in I} C^*(E_i,C^i)$, where $(E,C)= \varinjlim_{i\in I} (E_i, C^i)$ in the category $\FSGr$. Now it follows from the description of the
 direct limit in the category $\FSGr$ that $J_{(E,C)}= \varinjlim_{i\in I} J_{(E_i,C^i)}$. Indeed,  let $u,u'$ belong to the multiplicative subsemigroup of $C^*(E,C)$ generated by $(E^1)\cup (E^1)^*$.
 Then there is $i_0\in I$ such that all the edges appearing in the expressions of $u$ and $u'$ belong to $\varphi _{\infty,i_0}^1(E_{i_0}^1)$ (see \cite[Definition 8.4 and
 Proposition 3.3]{AG}).
 Here $\varphi _{\infty, i}\colon (E_i,C^i)\to (E,C)$ are  the canonical maps to the direct limit, for $i\in I$.
Hence there are $v,v'$ in the multiplicative subsemigroup of $C^*(E_{i_0}, C^{i_0})$ generated by $E_{i_0}\cup (E_{i_0})^*$  such that
 $$[e(u),e(u')] =  C^*(\varphi_{\infty, i_0}) ([e(v),e(v')]), $$
 and this implies that $J_{(E,C)}= \varinjlim_{i\in I} J_{(E_i,C^i)}$. This in turn implies that
 $$\mathcal O (E,C) = C^*(E,C) /J_{(E,C)} = \varinjlim _ {i\in I}\,   C^*(E_i, C^i)/J_{(E_i, C^i)} = \varinjlim_{i\in I}\,  \mathcal O (E_i, C^i)\, , $$
as desired.
\end{proof}

With these preliminaries, we can already obtain the description of $K_1$ of tame graph C*-algebras of finitely separated graphs. We still
will need further work to obtain the corresponding result for $K_0$.

\begin{theorem}
 \label{thm:mainK1fggrs}
 Let $(E,C)$ be a finitely separated graph. Then the natural projection map $\pi_{(E,C)} \colon C^*(E,C) \to \mathcal O (E,C)$ induces an isomorphism
$$K_1(\mathcal O (E,C)) \cong K_1(C^*(E,C)) \cong \ker (1_C-A_{(E,C)}).$$
   \end{theorem}

   \begin{proof}
   By \cite[Theorem 1.6]{AG}, $C^*(E,C)=\varinjlim _{\mathcal C} C^*(F,D)$, where $\mathcal C$ is the directed system of the finite complete subobjects of $(E,C)$
in the category $\FSGr$. By Proposition \ref{Ofunctor}, we have that $\mathcal O (E,C) = \varinjlim _{\mathcal C} \mathcal O (F,D)$. By using
Theorem \ref{thm:mainKifinitegrs}(2) and the continuity of $K_1$, we get
$$K_1(\mathcal O (E,C)) = \varinjlim_{\mathcal C} K_1(\mathcal O (F,D)) \cong \varinjlim_{\mathcal C} K_1( C^* (F,D)) = K_1(C^*(E,C)),$$
with the mapping $K_1(\pi_{(E,C)})$ inducing the isomorphism. The last part follows from \cite[Theorem 5.2]{AG2}.
\end{proof}

The correspondence $(E,C) \mapsto (\tilde{E}, \tilde{C})$ from \cite[Proposition 9.1]{AE} can be extended to a certain functor, which we describe below.

\begin{definition}
 \label{defi:BFSGr} The objects of the category $\BFSGr$ are all the bipartite finitely separated graphs. We stress here that this condition
includes that $r(E^1)= E^{0,0}$ and that $s(E^1)= E^{0,1}$ (see Definition \ref{def:bipartitesepgraph}). For objects $(E,C)$ and $(F,D)$ of $\BFSGr$,
the morphisms from $(E,C)$ to $(F,D)$ are the morphisms $\phi \colon E\to F$ of bipartite graphs (so that $\phi^0(E^{0,0})\subseteq F^{0,0}$ and
$\phi^0(E^{0,1})\subseteq F^{0,1}$, such that $\phi ^0$ is injective, and such that, for each $X\in C$ there is (a unique) $Y\in D$ such that
 $\phi ^1$ induces a bijection from $X$ onto $Y$.

 The category $\BFSGr$ is in fact a full subcategory of the category $\FSGr$. Indeed, if $(E,C), (F,D) \in \BFSGr$ and $\phi$ is a morphism in $\FSGr$
 from $(E,C)$ to $(F,D)$, then, for $v\in E^{0,0}$, there is $e\in E^1$ such that $r_E(e)= v$ and so $\phi ^0 (v)= r_F(\phi^1(e)) \in F^{0,0}$.
 Similarly, $\phi^0(E^{0,1})\subseteq F^{0,1}$. Hence, $\BFSGr$ is just the full subcategory of $\FSGr$ whose objects are the finitely separated graphs
 $(E,C)$ such that $E^0 = s(E^1) \sqcup r(E^1)$.

 We define the functor ${\bf B} \colon \FSGr \to \BFSGr$ by ${\bf B} ((E,C))= (\tilde{E}, \tilde{C})$, where $(\tilde{E},\tilde{C})$ is the bipartite
 separated graph associated to $(E,C)$ in \cite[Proposition 9.1]{AE}. We have that $\tilde{E}^{0,0}=V_0$ and $\tilde{E}^{0,1}= V_1$, where $V_0$ and $V_1$
 are disjoint copies of $E^0$, with bijections $E^0\to V^i$, $v\mapsto v_i$, and that
$\widetilde{E}^1$ is the disjoint union of a copy of $E^0$ and
a copy of $E^1$: $$\widetilde{E}^1= \{h_v\mid v\in E^0 \} \bigsqcup
\{e_0\mid e\in E^1 \},$$ with
$$
\tilde{r}(h_v)=v_0, \quad \tilde{s}(h_v)=v_1, \quad \tilde{r}(e_0)=
r(e)_0, \quad \tilde{s}(e_0)= s(e)_1, \qquad (v\in E^0, e\in E^1 )
.$$
For $v\in E^0$, and $X\in C_v$ put $\tilde{X} = \{e_0 : e\in X \}$. Then $\tilde{C}_{v_0}:= \{\tilde{X} : X\in C_v \} \sqcup \{{\bf h} _v \}$, where
${\bf h} _v : = \{ h_v \}$ is
a singleton set.

 For a morphism  $\phi \colon (E,C) \to (F,D)$ in $\FSGr$, the morphism ${\bf B} (\phi) \colon {\bf B} (E,C) \to {\bf B} (F,D)$ is defined by
 $${\bf B} (\phi)^0( v_i)= (\phi^0 (v))_i, \quad {\bf B} (\phi)^1(h_v) = h_{\phi^0(v)}, \quad {\bf B} (\phi)^1 (e_0)=
 \phi^1(e)_0, \,\,\,  (i=0,1, v\in E^0, e\in E^1).$$
 \end{definition}

We leave to the reader the proof of the following result, which is a straightforward extension of the arguments
in \cite[Proposition 9.1]{AE} and in Proposition \ref{Ofunctor}

\begin{proposition}
\label{prop:B-version-of-functoriality}
\begin{enumerate}
 \item[(a)] The category $\BFSGr$ is a full subcategory of $\FSGr$, closed under direct limits. Consequently the functors $C^*\colon \BFSGr\to \Cstaralg$
 and $\mathcal O \colon \BFSGr \to \Cstaralg$ are continuous.
 \item[(b)] There are natural isomorphisms of functors $\FSGr \to \Cstaralg $, $C^*\circ {\bf B}\cong {\bf M}_2 \circ C^*$,
 and $\mathcal O \circ {\bf B}\cong {\bf M}_2 \circ \mathcal O$, where ${\bf M}_2\colon \Cstaralg \to \Cstaralg$
 is the functor defined by ${\bf M}_2(A) = A\otimes M_2(\C)$.
 \item[(c)] Every object in $\BFSGr$ is the direct limit of its finite complete subobjects in $\BFSGr$.
 \end{enumerate}
\end{proposition}

In preparation for the next lemma, it is convenient to get a dynamical perspective on the C*-algebra homomorphism
$\mathcal O (E,C)\to \mathcal O (F,D)$, when $(E,C)$ is a complete subobject of the finite bipartite separated graph $(F,D)$.
Under this hypothesis, we are going to define an $(E,C)$-dynamical system on $\Omega : = \sqcup _{v\in E^0} \Omega(F,D)_v$.
For $v\in E^0$, set
$$\Omega_v: = \Omega (F,D)_v .$$
The sets $H_x$, for $x\in E^1$,  are the corresponding structural sets for $(F,D)$ and the homeomorphisms
$$\theta_x\colon \Omega_{s(x)} \longrightarrow H_x,\qquad x\in E^1$$
are also the structural homeomorphisms corresponding to $(F,D)$. Observe that $\Omega $ is a clopen subset of $\Omega (F,D)$.
By the universal property of the $(E,C)$-dynamical system
$\{ \Omega (E,C)_v \mid v\in E^0 \}$ there is a unique equivariant continuous map $\gamma \colon \Omega \to \Omega (E,C)$. It is not difficult to
describe this map in terms of the configurations used in \cite[Section 8]{AE}. Namely a point in $\Omega _v$, for $v\in E^0$, is given by a certain
subset of the free group $\mathbb F$ on $F^1$, with property (c) of \cite[page 783]{AE} at $g=1$ being satisfied with respect to the vertex $v$. If $\xi$ is such a configuration,
then $\gamma(\xi) $ is the configuration
on the free group on $E^1$, obtained by neglecting all the information which does not concern the graph $E$. In terms of the Cayley graphs, the map $\gamma$ consists
of deleting all the vertices and arrows which do not belong to $E^0$ and $E^1$ respectively. This is a well-defined map by the fact that $(E,C)$ is a complete
subgraph of $(F,D)$. The equivariant continuous map $\gamma \colon \Omega \to \Omega(E,C)$ is surjective and induces an equivariant injective unital homomorphism
$C(\Omega (E,C)) \to C(\Omega) \subseteq C(\Omega (F,D))$, and thus a homomorphism
$$\mathcal O (E,C) = C(\Omega (E,C)) \rtimes \mathbb F (E^1) \to C(\Omega (F,D)) \rtimes \mathbb F (F^1)= \mathcal O (F, D) .$$
Observe that this map is unital if and only if $E^0= F^0$.

The map $\gamma \colon \Omega \to \Omega (E,C)$ induces a map $\mathbb K (\gamma ) \colon \mathbb K (\Omega( E,C)) \to \mathbb K (\Omega) $, where $\mathbb K (\mathfrak X)$ denotes the field of open compact subsets on a
topological space $\mathfrak X$, where $\mathbb K (\gamma ) (K)= \gamma ^{-1} (K)$. Since the vertices in the complete multiresolution graphs of $(E,C)$ and $(F,D)$ provide a basis of open compact
subsets of the corresponding spaces $\Omega (E,C)$ and $\Omega (F,D)$, it is clear that the map $\mathbb K (\gamma )$ will have a significance with respect to these vertices. The exact connection is described below in Lemma
\ref{lem:injectivity-on-complements}.

To show this lemma we need first to introduce a new kind of maps between finite bipartite separated graphs, which is precisely the kind of maps that appear
when we study the maps $(E_n, C^n) \to (F_n, D^n)$ induced by a complete subobject $(E,C)\to (F,D)$ in the category $\BFSGr$. (Here $\{ (E_n, C^n)\} _n$ and $\{ (F_n, D^n) \}_n $
denote the canonical sequences of finite bipartite separated graphs associated to $(E,C)$ and $(F, D)$, respectively; see Construction \ref{cons:complete-multiresolution}(c).)  It is worth to
observe that these maps also induce C*-algebra homomorphisms between the respective
separated graph C*-algebras (see Lemma \ref{lem:locally-complete-induces}).

\begin{definition}
 \label{def:locally-complete} Let $(E,C)$ and $(F,D)$ be two finite bipartite separated graphs. A {\it locally complete} map $\pi^* \colon (E,C) \to (F,D)$ consists of a complete subobject $(G,L)$
 of $(F,D)$ and a graph homomorphism $\pi = (\pi^0, \pi^1) \colon (G,L)\to (E,C)$, such that:
 \begin{enumerate}
  \item $\pi^0 \colon G^0\to E^0$ and $\pi^1\colon G^1\to E^1$ are surjective maps.
  \item For each $X\in L$,  we have $\pi^1(X) \in C$. In particular, $\pi^1$ induces a (surjective) map $\tilde{\pi}\colon L \to C$, by $\tilde{\pi} (X) = \pi^1 (X)\in C$, for $X\in L$.
  \item For each $w\in G^{0,1}$, the map $\pi^1|_w \colon s_G^{-1} (w) \to s_E^{-1} (\pi^0(w))$ is a bijection.
  \item For each $v\in G^{0,0}$, the map $\tilde{\pi}|_v\colon L_v\to C_{\pi^0(v)}$ is a bijection.
 \end{enumerate}
\end{definition}

\begin{lemma}
 \label{lem:locally-complete-induces}
 Let $\pi^* \colon (E,C)\to (F,D)$ be a locally complete map between finite bipartite separated graphs. Then there is an induced $*$-homomorphism
 $C^*(\pi^*)\colon C^*(E,C)\to C^*(F,D)$. Moreover, there is
 a canonical locally complete map $\rho^*\colon (E_1, C^1)\to (F_1, D^1)$ such that the following diagram is commutative:
\begin{equation}
\label{eq:comdigram-for-locally-complete}
 \begin{CD}
C^*(E,C) @>C^*(\pi^*)>> C^*(F,D)\\
@V\phi(E,C)_0VV @VV\phi (F,D)_0V  \\
C^* (E_1, C^1) @>C^*(\rho^*)>> C^* (F_1, D^1)
\end{CD}\end{equation}
where $\phi (E,C)_0$ and $\phi (F,D)_0$ are the canonical surjective maps (cf. Theorem \ref{thm:algebras}).
 \end{lemma}

\begin{proof}
 Define $C^*(\pi^*)$ as follows. For $v\in E^0$ and $e\in E^1$, set
 $$ C^*(\pi^*) (v) = \sum _{w\in (\pi^0)^{-1}(v)} w,\qquad C^*(\pi^*) (e) = \sum _{f\in (\pi^1)^{-1} (e)} f . $$
 It is easy to check that relations (V) and (E) are preserved by $C^*(\pi^*)$. To show that relation (SCK1) is preserved, consider $e,f\in X$, where $X\in C_v$.
 Assume first that $e\ne f$. Take $g, h\in G^1$ such that $\pi^1(g)= e$ and $\pi^1(h)= f$. If $r(g) \ne r(h)$, then $g^*h= 0$. If $r(g)= r(h)$, then $g\ne h$ and $g,h$ belong to the same
 element of $L$, by condition (4) in Definition \ref{def:locally-complete}. (Indeed, if $g\in Y\in L_{r(g)}$ and $h\in Z\in L_{r(g)}$, then $\tilde{\pi}|_{r(g)}(Y)= X= \tilde{\pi}|_{r(g)}(Z)$, and so
 $Y=Z$ by the injectivity of $\tilde{\pi}|_{r(g)}$.) Therefore $g^*h= 0$. It follows that
 $$C^*(\pi^*) (e^*f) = ( \sum _{\pi^1(g)= e} g^*)(\sum_{\pi^1(h)=f} h) = 0.$$
 Now assume that $e=f$. By condition (3) in Definition \ref{def:locally-complete}, for each $w\in (\pi^0)^{-1}(s(e))$ there is a unique $h_w\in s^{-1}(w)$ such that $\pi^1(h_w)= e$.
 If $r(h_{w_1}) = r(h_{w_2})$ for $w_1,w_2\in (\pi^0)^{-1} (s(e))$, then it follows from the same argument as before that $h_{w_1}$ and $h_{w_2}$ belong to the same element of
 $L$. It follows that $h_{w_1}^*h_{w_2} =\delta _{w_1,w_2} w_1$ for all $w_1,w_2\in (\pi^0)^{-1} (s(e))$, and thus
 $$C^*(\pi^*) (e^*e) = (\sum_{w_1\in (\pi^0)^{-1}(s(e))} h_{w_1}^*) (\sum_{w_2\in (\pi^0)^{-1}(s(e))} h_{w_2}) = \sum _{w\in (\pi^0)^{-1}(s(e))} w = C^*(\pi^*) (s(e)), $$
 as desired.

 Now we check that (SCK2) is preserved by $C^*(\pi^*)$. Take $v\in E^{0,0}$ and $X\in C_v$. Let $g,h\in G^1$ be such that $\pi^1(g) = e =\pi^1(h)$, where $e\in X$. If $s(g)\ne s(h)$, then $gh^*=0$.
 If $s(g)= s(h)$, then by condition (3) in Definition \ref{def:locally-complete} we have that $g=h$. It follows that $C^*(\pi^*) (ee^*) = \sum _{g\in (\pi^1)^{-1} (e)} gg^*$.
 Now, it follows from conditions (2) and (4) in Definition
 \ref{def:locally-complete} that for each $w\in (\pi^0)^{-1} (v)$ there is a unique $Y_w\in L_w$ such that $(\pi^1)^{-1} (X) \cap r^{-1}(w) = Y_w$. Hence, we get
 \begin{align*}
C^*(\pi^*) & \Big( \sum _{e\in X} ee^* \Big)    = \sum _{e\in X} \sum _{g\in (\pi^1)^{-1} (e)} gg^* = \sum _{g\in (\pi^1)^{-1} (X)} gg^* \\
& =
 \sum _{w\in (\pi^0)^{-1} (v)} \Big( \sum _{g\in Y_w} gg^*\Big)
 =  \sum _{w\in (\pi^0)^{-1} (v)} w= C^*(\pi^*) (v) \, ,
 \end{align*}
 as desired.

 We now show the statement about the associated separated graphs $(E_1,C^1)$ and $(F_1, D^1)$. We first define a complete subobject $(G_1, L^1)$ of $(F_1, D^1)$. Set
 $G_1^{0,0}= G^{0,1}$ and $G_1^{0,1} = r_2^{-1}(G^{0,0})$. In other words, $v\in G_1^{0,1}$ if and only if there is $u\in G^{0,0}$ such that
$v= v(x_1,\dots , x_l)$, where $x_i\in X_i$ and $D_u= \{ X_1,\dots , X_l\}$.
Now for $w\in G_1^{0,0}= G^{0,1}$, define
$$L^1_w= \{ X(x) \mid x\in G^1 \}, \qquad  G_1^1= \bigsqcup _{w\in G_1^{0,0}} L_w^1 .$$
Clearly $(G_1,L^1)$ is a complete subobject of $(F_1,D^1)$.

Now we define the graph homomorphism $\rho = (\rho ^0, \rho^1)\colon G_1\to E_1$. Define $\rho ^0(w) = \pi^0(w)$ for $w\in G_1^{0,0}= G^{0,1}$.
Now, for $u\in G^{0,0}$, set $D_u= \{X_1,\dots , X_k, X_{k+1},\dots ,X_l \}$, where $L_u = \{ X_1,\dots , X_k \}$. Then define $\rho ^0$ on an element
$v= v(x_1,\dots ,x_k, x_{k+1},\dots ,x_l)$, with $x_i\in X_i$, $i=1,\dots ,l$,  by
$$\rho^0(v(x_1,\dots, x_k, x_{k+1}, \dots ,x_l)) =  v(\pi^1(x_1),\dots ,\pi^1(x_k)) \in E_1^{0,1}. $$
Note that this is well-defined because, by conditions (2) and (4), we have that $C_{\pi^0(u)}= \{ \pi^1(X_1),\dots , \pi^1(X_k) \}$.

 Now we define $\rho^1$. An element in $G_1^1$ is of the form $\alpha^{x_i}(x_1,\dots ,\widehat{x_i}, \dots, x_k, x_{k+1}, \dots ,x_l )$, where $x_1,\dots ,x_l$ (and $X_1,\dots ,X_l$)
 are as above.
 For such an element, put
 $$\rho^1 (\alpha^{x_i}(x_1,\dots ,\widehat{x_i},\dots, x_k, x_{k+1}, \dots ,x_l )) = \alpha^{\pi^1(x_i)}(\pi^1(x_1),\dots , \widehat{\pi^1(x_i)}, \dots \pi^1(x_k)).$$
Clearly $\rho$ is a graph homomorphism. Finally, we have to check conditions (1)-(4) in Definition \ref{def:locally-complete} for $\rho$.

(1) Let $v(x_1,\dots ,x_k)\in E_1^{0,1}$, where $x_i\in X_i$ and $C_u= \{X_1,\dots , X_k\}$ for some $u\in E^{0,0}$. Since $\pi^0$ is surjective, there is $u'\in G^{0,0}$
such that $\pi^0 (u') = u$. Now, by conditions (2) and (4)  (for $\pi$),  we can write $D_{u'} = \{ Y_1,\dots , Y_k, Y_{k+1}, \dots , Y_l \}$ and $L_{u'} = \{ Y_1,\dots , Y_k\}$,
with $\pi^1( Y_i)= X_i$
for $i=1,\dots , k$. Take $y_i\in Y_i$ such that $\pi^1( y_i) = x_i$, $i=1,\dots , k$,  and take any $y_j\in Y_j$ for $j=k+1,\dots ,l$. Then
$$\rho ^0(v(y_1, \dots , y_k, y_{k+1},\dots , y_l)) = v(\pi^1(y_1), \dots , \pi^1( y_k)) = v(x_1,\dots , x_k) .$$
This shows that $\rho^0$ is surjective. For $i=1,\dots ,k$, we also get
$$\rho^1 (\alpha^{y_i}( y_1,\dots , \widehat{y_i},\dots, y_k,y_{k+1},\dots, y_l)) = \alpha ^{x_i}(x_1,\dots , \widehat{x_i}, \dots, x_k)\, ,$$
which shows that $\rho^1$ is also surjective.

(2) If $X\in L_1$, then there is $u\in G^{0,0}$ with $D_u=\{ X_1,\dots , X_k, X_{k+1},\dots , X_l \}$ and $L_u = \{ X_1,\dots , X_k \}$ such that
$X=X(x_i)$ for some $i$ with $1\le i\le k$. By the definition of $\rho^1$ and conditions (2) and (4) for $\pi^1$, we get that $\rho^1 (X) = X(\pi^1(x_i))$.

(3) Let $v= v(x_1,\dots , x_k,x_{k+1},\dots , x_l)$ be a vertex in $G_1^{0,1}$, where the notation is as before. Then
$$s_{G_1}^{-1} (v) = \{ \alpha ^{x_i} (x_1,\dots , \widehat{x_i},\dots , x_k,x_{k+1},\dots , x_l ) \mid i=1,\dots , k \}\, ,$$
so that it is clear that $\rho^1$ induces a bijection $\rho^1|_v\colon s_{G_1}^{-1} (v) \to s_{E_1}^{-1} (\rho^0(v))$.

(4) Let $w\in G_1^{0,0} = G^{0,1}$. Then the elements of $L^1_w$ are in bijective correspondence with the elements of $s_{G}^{-1}(w)$. If $x\in G^1$ is one of such vertices, then the corresponding element of
$L^1_w$ is $X(x)$, and $\tilde{\rho} (X(x)) = X(\pi^1 (x))$. Since $\pi^1$ establishes a bijection between $s_G^{-1}(w)$ and $s_E^{-1} (\pi^0 (w))$, we see that
$\tilde{\rho}$ establishes a bijection from $L^1_w$ onto $C^1_{\rho^0(w)}$, as desired.
 \end{proof}

 \begin{corollary}
  \label{cor:stabilize-inclusions}
  Let $(F,D)$ be a finite bipartite separated graph, and let $(E,C)$ be a complete subobject of $(F, D)$ in $\BFSGr$. Let $\{ (E_n,C^n)\}$ and $\{ (F_n,D^n)\}$ be the canonical sequences of
  finite bipartite separated graphs associated to $(E,C)$ and $(F,D)$ respectively. Then there are canonical locally complete maps $\pi_n^*\colon  (E_n,C^n)\to (F_n, D^n)$ such that
  $C^*(\pi_{n+1}^*)\circ \phi(E,C)_n = \phi(F,D)_n\circ C^*(\pi^*_n)$ for all $n\ge 0$. Consequently, if $\iota \colon (E,C)\to (F,D)$ is the inclusion map, and $\mathcal O (\iota)
  \colon \mathcal O (E,C) \to \mathcal O (F,D)$ is the induced $*$-homomorphism, then $\mathcal O (\iota) = \varinjlim_n C^*(\pi_n^*)$.
 \end{corollary}

 \begin{proof}
  Use Lemma \ref{lem:locally-complete-induces} and induction, starting with the natural map $\iota \colon (E,C)\to (F,D)$, which is obviously a locally complete map.
 \end{proof}

Using suitable orderings we will be able to determine a  {\it canonical complement} $H_{(E,C)}$ of $K_0(C^*(E,C))$ in $K_0 (\mathcal O (E,C))$, for each finite bipartite separated
graph $(E,C)$.

\begin{definition}
 \label{def:ordered-fsbipsepargraph}
Let $(E,C)$ be a bipartite finitely separated graph. An {\it order} in $(E,C)$ is given by the following data:
\begin{enumerate}
 \item A total order in each of the sets $C_v$, for $v\in E^{0,0}$.
 \item  A total order in each of the sets $s_E^{-1}(w)$, for $w\in E^{0,1}$.
 \item A total order in each of the sets $X$, for $X\in C$.
\end{enumerate}
It is clear that every bipartite finitely separated graph can be endowed with an order. When this is given we refer to $(E,C)$ as
an {\it ordered} separated graph. If $(E,C)$ is ordered, each complete subobject $(F,D)$ of $(E,C)$ in $\BFSGr$ inherits an order, defined by restricting the
corresponding total orderings.
 \end{definition}

\begin{notation}
\label{notation:HsubEC}
 Let $(E,C)$ be an ordered finite bipartite separated graph. Then the proof of Theorem \ref{thm:K0C(E,C)} and Lemma \ref{lem:K0multires}
  give a canonical
 complement of $K_0(C^*(E,C))$ in $K_0(C^*(E_1, C^1))$, namely the group
$\Z^{W_2}$, where $W_2$ is the set of vertices of $E_1^{0,1}$ of the form $v(x_1, \cdots , x_k)$, where $x_i\in X_i$, $C_u=\{X_1,\dots , X_k\}$
for some $u\in E^{0,0}$, and at least two different elements $x_i$ and $x_j$ are not the first elements in the respective sets $X_i$ and $X_j$ in the given order
on them. The choice of a given order in each of the sets $X\in C^n$, for all the sets $C^n$ appearing in the canonical sequence of finite bipartite separated graphs
$\{(E_n,C^n)\}$ associated to $(E,C)$ will thus,
by Theorem \ref{thm:K0C(E,C)},
give a {\it canonical complement} $H_{(E,C)}$ of $K_0(C^*(E,C))$ in $K_0 (\mathcal O (E,C))$. Indeed, we can inductively define an order on each of the finite bipartite separated graphs
$(E_n, C^n)$, as follows. Assume that, for some $n\ge 0$, an order has been defined on $(E_n, C^n)$, and let us define the order on $(E_{n+1}, C^{n+1})$.
For $v\in E_{n+1}^{0,0}= E_n^{0,1}$, we have that $C_v^{n+1}$ is in bijective correspondence with $s_{E_n}^{-1}(v)$ (through $X(x)\leftrightarrow x$).
Define the total order in $C_v^{n+1}$ as the order
induced by this bijection. For $v\in E_{n+1}^{0,1}$, we have $v= v(x_1,\dots , x_k)$, where $u\in E_n^{0,0}$, with $C^n_u=\{X_1,\dots , X_k \}$, and $x_i\in X_i$ for $i=1,\dots , k$.
(Here we are assuming that $X_1<X_2<\cdots <X_k$ in the given total order on $C^n_u$.) Now note that
$$s_{E_{n+1}}^{-1} (v)= \{\alpha^{x_i}(x_1, \dots, x_{i-1}, x_{i+1}, \dots , x_k) : i=1,\dots ,k \}.$$
We define the total order in $s_{E_{n+1}}^{-1} (v)$ by setting $\alpha^{x_i}(x_1, \dots, \hat{x}_i, \dots , x_k) < \alpha^{x_j}(x_1, \dots, \hat{x}_j, \dots , x_k)$
if and only if $i<j$.
Finally, let $X$ be an element of $C^{n+1}$.
Then there is $u\in E_n^{0,0}$, with $C^n_u=\{X_1,\dots , X_k \}$, and $x_i\in X_i$ for some $i=1,\dots , k$, such that $X= X(x_i)$. Recall that
$$X(x_i) = \{ \alpha^{x_i}(x_1,\dots, x_{i-1}, x_{i+1}, \dots , x_k) : x_j\in X_j, j\ne i \}\cong X_1\times \cdots \times X_{i-1}\times X_{i+1}\times \cdots \times X_k\, ,$$
so we take the left lexicographic order on $X(x_i)$.

This gives a canonical choice of sets $W_2,W_3, \dots $ and thus a canonical choice of a complement
$H_{(E,C)}:=\bigoplus_{k=2}^{\infty} \Z^{W_k}$
of $K_0(C^*(E,C))$ in $K_0(\mathcal O (E,C))$, so that
\begin{equation}
 \label{eq:can-choice-of-compl}
 K_0(\mathcal O (E,C) ) = K_0 (C^*(E,C)) \oplus H_{(E,C)} .
\end{equation}
\end{notation}

\medskip

\begin{lemma}
 \label{lem:injectivity-on-complements}
 Let $(F,D)$ be an ordered finite bipartite separated graph, and let $(E,C)$ be a complete subobject of $(F, D)$ in $\BFSGr$, endowed with the induced order.
 Let $\varphi \colon K_0(\mathcal O (E,C)) \to K_0(\mathcal O (F,D))$ denote the map induced by the inclusion $\iota \colon (E,C)\to (F,D)$. Then the restriction of $\varphi$ to
 $H_{(E,C)}$ is injective, and $\varphi (H_{(E,C)}) \subseteq H_{(F,D)}$.
\end{lemma}

\begin{proof}
 By the proof of Theorem \ref{thm:K0C(E,C)} and Corollary \ref{cor:stabilize-inclusions}, it suffices to show inductively that, for each $n\ge 1$,
 the induced map $C^*(\pi_n^*) \colon C^*(E_n, C^n) \to C^*(F_n, D^n)$ sends each projection coming from $W_{n+1}$
 to an orthogonal sum of projections coming from $W_{n+1}'$, where $W_{n+1}$ corresponds to $(E_n,C^n)$ and $W_{n+1}'$ corresponds to $(F_n, D^n)$.
 The injectivity of $\varphi |_{H_{(E,C)}} $ follows then from the fact that $C^*(\pi_n^*)$ sends projections corresponding to distinct vertices
 of $E_n$ to orthogonal projections of $C^*(F_n,D^n)$ (see Lemma \ref{lem:locally-complete-induces}).
  In order to show this, it is enough to show,  by Lemma \ref{lem:locally-complete-induces} and induction, that the result holds
 for the first terms $(E_1, C^1)$, $(F_1, D^1)$ of the canonical
 sequences of finite bipartite separated graphs associated to $(E,C)$ and $(F,D)$ respectively, where $\pi^*\colon (E,C)\to (F,D)$ is a certain locally complete map.
 Concretely we will show the following statement:

 {\it Claim:} Let $\pi^*\colon (E,C)\to (F,D)$ be a locally complete map, and let $\rho^*\colon (E_1,C^1)\to (F_1, D^1)$ be the corresponding locally complete map, as defined
 in the proof of Lemma \ref{lem:locally-complete-induces}. Assume that the following condition holds:
 $\pi^1$ sends the first element of each $Y$ in $L$ to the first element of $\tilde{\pi}(Y)\in X$. Then $C^*(\rho^1)$ sends
 each projection coming from $W_2$ to a projection in $C^*(F_1, D^1)$ which is an orthogonal sum of projections coming from $W_2'$. Moreover, the map $\rho ^*$ has
 the same property as $\pi^*$, that is, it sends the first element of each $Y\in L_1$ to the first element of $\tilde{\rho} (Y)\in C_1$.

\medskip

{\it Proof of Claim:}
The set $W_2$ above is the set of projections of the form  $v= v(x_1, \dots , x_k)$,
 where $x_i\in X_i$, $C_u=\{ X_1,\dots , X_k\}$, and at least for two different indices $j,t$ we have that $x_j$ and $x_t$ are not the first elements of $X_j$ and $X_t$ respectively
 (see the proofs of Theorem \ref{thm:K0C(E,C)} and Lemma \ref{lem:K0multires}).
The set $W_2'$ is the analogous set of projections in $C^*(F,D)$.

 For $v= v(x_1,\dots , x_k)\in W_2$, we have
 $$C^*(\rho^*)(v) = \sum  v(y_1,\dots , y_k, y_{k+1}, \dots , y_l) \, ,$$
 where the sum is extended over all $(y_1,\dots , y_l)\in Y_1\times \cdots \times Y_l$, where $D_{u'} = \{Y_1, \dots , Y_l \}$ and $L_{u'} = \{ Y_1,\dots ,Y_k \}$, where
 $u'$ ranges over all the vertices in $G$ such that $\pi^0 (u') = u$,  and $\pi^1(y_i) = x_i$ for all $i=1,\dots , k$. (Note that here the index $l$ may depend on $u'$.)

 Now by the hypothesis on $\pi^1$, we have that $y_j$ is not the first element of $Y_j$ and $y_t$ is not the first element of $Y_t$, showing that each $ v(y_1,\dots , y_k, y_{k+1}, \dots , y_l)$
 belongs to $W_2'$.

Finally we check that $\rho^1$ has the same property as $\pi^1$. Take $Y\in L^1$. Then there is $u\in G^{0,0}$, with $D_u=\{X_1,\dots, X_k, X_{k+1}, \dots , X_l \}$ and $L_u= \{X_1,\dots, X_k \}$
such that $Y= X(x_i)$ for some $x_i\in X_i$ with $1\le i\le k$. The first element of $Y$ is thus the element
$$e= \alpha^{x_i}(x_1,\dots ,\widehat{x_i}, \dots, x_k, x_{k+1},\dots , x_l), $$
where, for each $j\ne i$, $x_j$ is the first element of $X_j$. Consequently, by the hypothesis on $\pi^1$, the element $\pi^1(x_j)$ is the first element of $\tilde{\pi}(X_j)$, for $j\ne i$ and
$j\in \{1,\dots , k\}$. Therefore
$$\rho^1 (e) = \alpha ^{\pi^1(x_i)} ( \pi^1(x_1) , \dots, \widehat{\pi^1(x_i)} , \dots,  \pi^1(x_k) ) \, ,$$
which is the first element of $X(\pi^1(x_i)) = \tilde{\rho} (Y)$. \qed

Note that the hypothesis on $\pi^1$ is trivially satisfied in the base case, that is, in the case where $(E,C)$ is a complete subobject of $(F,D)$, Indeed, in that case $(G,L)= (E,C)$ and
$\pi$ is the identity. Therefore, the Claim gives the desired result by induction, using Lemma \ref{lem:locally-complete-induces}.
   \end{proof}

\begin{theorem}
 \label{thm:mainK0fggrs}
 Let $(E,C)$ be an ordered bipartite finitely separated graph and let $\mathcal C$ be the directed set of finite complete subobjects of $(E,C)$ in $\BFSGr$.
 For complete subobjects $(F,D), (F',D')$ of $(E,C)$,  with $(F,D)\le (F',D')$, let $\varphi_{(F',D'),(F,D)}\colon K_0(\mathcal O (F,D))\to K_0(\mathcal O (F',D'))$ be the natural
 map. Write $K_0(\mathcal O (F,D))= K_0(C^*(F,D))\oplus H_{(F,D)}$ for each $(F,D)\in \mathcal C$, where $H_{(F,D)}$ is the canonical complement associated to the induced order on $(F,D)$,
 as defined in Notation \ref{notation:HsubEC}. Then the following properties hold:
 \begin{enumerate}
  \item For $(F,D), (F',D')\in \mathcal C$ with $(F,D)\le (F',D')$, the map $\varphi_{(F',D'), (F,D)}$ induces an injective homomorphism from
  $H_{(F,D)}$ to $H_{(F', D' )}$.
\item We have
$$K_0 (\mathcal O (E,C)) \cong K_0(C^*(E,C))\bigoplus H\ \cong \coker (1_C-A_{(E,C)}) \bigoplus H \, ,$$
where $H = \varinjlim_{(F,D)\in \mathcal C} H_{(F,D)}$. In particular $H$ is a torsion-free group, and the maps $\varphi _{(E,C), (F,D)}|_{H_{(F,D)}}$ are injective for all
$(F,D)\in \mathcal C$.
   \end{enumerate}
 \end{theorem}

   \begin{proof}
The decomposition  $K_0(\mathcal O (F,D))= K_0(C^*(F,D))\oplus H_{(F,D)}$ for each $(F,D)\in \mathcal C$ is described in Notation \ref{notation:HsubEC}.
(1) follows from Lemma \ref{lem:injectivity-on-complements}, and (2) follows from Proposition \ref{prop:B-version-of-functoriality}, the continuity of $K_0$ and (1).
\end{proof}

\begin{theorem}
\label{thm:finalK0}
Let $(E,C)$ be a finitely separated graph. Then $K_0 (\mathcal O (E,C))= K_0(C^*(E,C))\oplus H$, where $H$ is a torsion-free group.
\end{theorem}

\begin{proof}
 This follows from \cite[Proposition 9.1]{AE} and Theorem \ref{thm:mainK0fggrs}.
\end{proof}


\begin{thebibliography}{12}

\bibitem{Aone-rel} P. Ara, \emph{Purely infinite simple reduced
C*-algebras of one-relator separated graphs}, J. Math. Anal. Appl.
{\bf 393} (2012),  493--508.

\bibitem{AE} P. Ara, R. Exel, \emph{Dynamical systems associated to
separated graphs, graph algebras, and paradoxical decompositions},
Adv. Math. {\bf 252} (2014), 748--804.

\bibitem{AEK} P. Ara, R. Exel, T. Katsura, \emph{Dynamical systems
of type $(m,n)$ and their C*-algebras}, Ergodic Theory Dynam.
Systems  {\bf 33} (2013), 1291--1325.

\bibitem{AG2} P. Ara, K. R. Goodearl, \emph{C*-algebras of
separated graphs}, J. Funct. Anal. {\bf 261} (2011), 2540--2568.

\bibitem{AG} P. Ara, K. R. Goodearl, \emph{Leavitt path algebras of
separated graphs}. J. reine angew. Math. {\bf 669} (2012), 165--224.


\bibitem{Black} {\sc B. Blackadar}, ``K-Theory for Operator
Algebras'', Second Edition, M.S.R.I. Publications, vol. 5, Cambridge
Univ. Press, Cambridge, 1998.

\bibitem{Brown} L. G. Brown, \emph{Ext of certain free product $C^{\ast}
$-algebras}, J. Operator Theory {\bf 6} (1981), 135--141.

\bibitem{CET} T. M. Carlsen, S. Eilers, M.  Tomforde, \emph{Index maps in the K-theory of graph algebras}, J. K-Theory {\bf 9} (2012), 385--406.


\bibitem{ExelPJM} R. Exel, \emph{Partial representations and amenable Fell bundles over
free groups}, Pacific J. Math. {\bf 192} (2000), 39--63.

\bibitem{KN} D. Kerr, P. W. Nowak, \emph{Residually finite actions
and crossed products}, Ergodic Theory and Dynamical Systems {\bf 32}
(2012), 1585--1614.


\bibitem{McCla1} K. McClanahan, \emph{$C\sp *$-algebras generated by elements of a
unitary matrix}, J. Funct. Anal. {\bf 107} (1992), 439--457.

\bibitem{McCla2} K. McClanahan, \emph{$K$-theory and ${\rm
Ext}$-theory for rectangular unitary $C\sp *$-algebras}, Rocky
Mountain J. Math. {\bf 23} (1993), 1063--1080.

\bibitem{McCla3} K. McClanahan, \emph{Simplicity of reduced amalgamated products of
C*-algebras}, Canad. J. Math. {\bf 46} (1994), 793--807.

\bibitem{McCla4} K. McClanahan,  \emph{K-theory for partial crossed products by
discrete groups}, J. Funct. Anal. \textbf{130} (1995), 77--117.


\bibitem{Raeburn} I. Raeburn, Graph algebras. CBMS Regional Conference Series in
Mathematics, 103. Published for the Conference Board of the
Mathematical Sciences, Washington, DC; by the American Mathematical
Society, Providence, RI, 2005.

\bibitem{RaeSzy} I. Raeburn and W. Szyma\'nski ,
\emph{Cuntz-Krieger algebras of infinite graphs and matrices}, Trans. Amer. Math. Soc. \textbf{356} (2004), 39--59.


 \bibitem{rordam} M. R\o rdam, F. Larsen, N.J. Laustsen,
``An Introduction to $K$-Theory for $C^*$-Algebras", Cambridge
University Press, LMS Student Texts 49, 2000.

\bibitem{RS} M. R\o rdam, A. Sierakowski, \emph{Purely infinite
C*-algebras arising from crossed products}, Ergodic Theory and
Dynamical Systems {\bf 32} (2012), 273--293.


\bibitem{Thomsen} K. Thomsen, \emph{On the $KK$-theory and the $E$-theory
of amalgamated free products of C*-algebras}, J. Func.
Anal. \textbf{201} (2003), 30--56.



\end{thebibliography}
\end{document}